\documentclass[12pt]{ociamthesis}  

\usepackage{amssymb}
\usepackage[hyphens,spaces]{url}
\usepackage[english]{babel}
\usepackage{amsthm}
\usepackage{enumerate} 
\usepackage{mathtools}
\usepackage{tikz-cd}
\usepackage{tabularx}
\usepackage{setspace}
\newcommand{\R}{\mathbb{R}}
\newcommand{\C}{\mathbb{C}}
\newcommand{\Q}{\mathbb{Q}}
\newcommand{\F}{\mathbb{F}}
\newcommand{\Z}{\mathbb{Z}}
\newcommand{\N}{\mathbb{N}}
\newcommand{\ideal}{\vartriangleleft}
\newcommand{\idealinv}{\vartriangleright}
\renewcommand{\H}{\mathbb{H}}

\usepackage{hyperref}
\hypersetup{
    colorlinks,
    citecolor=black,
    filecolor=black,
    linkcolor=black,
    urlcolor=black
}
\usepackage{afterpage}
\usepackage{imakeidx}
\makeindex
\newcommand\blankpage{%
    \null
    \thispagestyle{empty}%
    \addtocounter{page}{-1}%
    \newpage}
    
\usepackage{etoolbox}
\makeatletter
\patchcmd{\l@chapter}{1.0em}{0.5em}{}{}
\makeatother
    
\title{The Brauer Group of Rational Numbers}   

\college{St Catherine's College}  
\author{Haiyu Chen}

\renewcommand{\submittedtext}{A dissertation submitted for the degree of}
\degree{M.Sc. in Mathematics and Foundations of Computer Science}     
\degreedate{Trinity Term 2019}         

\newtheorem{theorem}{Theorem}[section]
\newtheorem{corollary}[theorem]{Corollary}
\newtheorem{lemma}[theorem]{Lemma}
\newtheorem{proposition}[theorem]{Proposition}
\newtheorem{definition}[theorem]{Definition}
\newtheorem*{remark}{Remark}
\newcommand{\defeq}{\vcentcolon=}

\begin{document}

\baselineskip=18pt plus1pt

\setcounter{secnumdepth}{3}
\setcounter{tocdepth}{3}
\hyphenpenalty=10000
\binoppenalty=\maxdimen
\relpenalty=\maxdimen

\maketitle                  

\newpage
\afterpage{\blankpage}
\begin{romanpages}          
\chapter*{Abstract}\addcontentsline{toc}{chapter}{Abstract}\thispagestyle{empty}

    In this project, we will study the Brauer group that was first defined by R. Brauer. The elements of the Brauer group are the equivalence classes of finite dimensional central simple algebra. Therefore understanding the structure of the Brauer group of a field is equivalent to a complete classification of finite dimensional central division algebras over the field. One of the important achievements of algebra and number theory in the last century is the determination of $Br(\Q)$, the Brauer group of rational numbers. The aim of this dissertation is to review this project, i.e., determining $Br(\mathbb{Q})$. 
    
    There are three main steps. The first step is to determine $Br(\mathbb{R})$, the Brauer group of real numbers. The second step is to identify $Br(k_\nu)$, the Brauer group of the local fields. The third step is to construct two maps $Br(\mathbb{Q})\to Br(\mathbb{R})$ and $Br(\mathbb{Q})\to Br(\mathbb{Q}_p)$ and to use these two maps to understand $Br(\mathbb{Q})$. This dissertation completed the first two steps of this enormous project.
    
    To the author's knowledge, in literature there is no document including all the details of determining $Br(\Q)$ and most of them are written from a advanced perspective that requires the knowledge of class field theory and cohomology. The goal of this document is to develop the result in a relatively elementary way. The project mainly follows the logic of the book \cite{jacobson2009}, but significant amount of details are added and some proofs are originated by the author, for example, \ref{1.9}, \ref{1.4.2}(ii), \ref{4.2.6}, and maximality and uniqueness of \ref{5.5.12}.
\newpage
\afterpage{\blankpage}

\chapter*{Acknowledgements}\addcontentsline{toc}{chapter}{Acknowledgements}

I would like to express my deepest gratitude to my supervisor, Prof. Nikolay Nikolov, for his continuous guidance throughout this project. I appreciate him taking time to think about this project that is perfect for my mathematical background while I can also learn new techniques in algebraic number theory and practice those in the course \emph{Non-commutative rings}. This project also familiarize me with Galois theory, algebraic number theory and some metric spaces and I can see my progress in mathematical maturity. Thanks to the guidance of Prof. Nikolov, I am able to overcome various difficulties that I faced for the first time in mathematical research. Apart from that, he always makes this project flexible to me, which has stimulated my greatest potential. His unique perspectives make this project accessible and enjoyable.\\ \par

Finally, I am gratefully indebted to my parents and friends who warmly support me spiritually whenever I need help. I would like to give a special thanks to my friend Deepak who discussed with me how mathematical research was like and give me emotional support when I got stuck.\\ \par

This accomplishment would not have been possible without them. Thank you.

\tableofcontents            
\addtocontents{toc}{\protect\enlargethispage{\baselineskip}}

\chapter*{Introduction}\addcontentsline{toc}{chapter}{Introduction}

    The aim of this dissertation is to study the Brauer group of rational numbers, which is one of the most important achievement in 1930's. The elements of the Brauer group over a base field are equivalent classes of finite dimensional division algebras whose centers coincide with the base field. The determination of the $Br(k)$ is equivalent to classifying all finite dimensional central simple algebras over the base field $k$.

    Roughly speaking, the larger the field, the smaller the Brauer group. For example, in chapter $2$ we will see the Brauer group of an algebraically closed field is trivial. In chapter $3$ we will see the Brauer group of real number is isomorphic to $\Z/2\Z$, the cyclic group of order $2$. It turns out in chapters 4-5 that the Brauer group of local field is isomorphic to $\Q/\Z$. Hence, the Brauer group of rational number is one of the "largest" Brauer groups one may get.
    
    The project of $Br(\Q)$ has three main phases. The first phase is to determine $Br(\R)$. The second phase is to identify $Br(k_\nu)$. The reason that this two steps are important is that the elements of $Br(\Q)$ can be studied via all completions of $\Q$. In some sense the "only" possible completions of $\Q$ are of two kinds, the finite, which is a local field $k_\nu$, or infinite, which is $\R$. Then in the third phase we can combine the results by studying the maps $Br(\Q)\twoheadrightarrow$ $Br(\R)$, and $Br(\Q)\twoheadrightarrow Br(\Q_p)$ and by the exact sequence \cite{stackexchangeBrauerGroup}
    \begin{equation*}
        0\to Br(\Q)\to \Z/2\Z \oplus_\nu \Q/\Z\to \Q/\Z\to 0.
    \end{equation*}
    
    The focus of this dissertation is the calculation of $Br(\R)$ and $Br(k_\nu)$ for local fields $k_\nu$. The first chapter discusses some classical results of finite dimensional central simple algebra. They are representatives of elements in the Brauer group. The second chapter defines the Brauer group, and gives that over algebraically closed field, the Brauer group is trivial. The third chapter calculate $Br(\R)$, as the title indicates. Chapter $4$ introduces cyclic algebras, which play an important role in determination of $Br(k)$ where $k$ is a global or local field. A consequence of Albert-Brauer-Hasse-Neother theorem is that a division algebra of any degree over $k$ is a cyclic algebra \cite{wikiBrauergroup}, or every central simple algebra over an algebraic number field is cyclic \cite{pierce1982associative}. Chapter $5$ introduces valuation theory, which is essential in determining the finite field extensions of local fields, and finally contributes to the determination of the Brauer group of local fields.
    
    As a part of classical field theory, an application of this result is to generalize the reciprocity law. See chapter VII in \cite{cassels1967algebraic}.

    It worths to mention that there are other equivalent definitions of Brauer groups. This dissertation uses a concrete approach to determine the Brauer groups. One of the equivalent ways is to use Galois cohomology. In chapter $4$ (corollary \ref{4.1.12}) we set up an isomorphism $Br(L/k)\cong H^2(G,L^\times).$ In fact, we can express Brauer group in an arbitrary field in terms of Galois cohomology as $$Br(k)\cong H^2(k,G_m)$$ where $G_m$ is the multiplicative group, viewed as an algebraic group over $k$ \cite{wikiBrauergroup}.

\newpage
\chapter*{Notations}
\begin{table}[h]
\centering
    \bgroup
    \def\arraystretch{1.4}
    \begin{tabularx}{\textwidth}{l c X}
         $Z(A)$         &:& The center of an ring $A$.\\
         $M_n(A)$       &:& The Matrix ring of $m\times n$ matrices.\\
         $A^{op}$       &:& The opposite ring $(A,+,\ast_{op})$ of the ring $(A,+,\cdot_{A})$ with multiplication defined by $a\ast_{op} b=b\cdot_{A} a$.\\
         $End({}_A V)$  &:& $A$-linear endomorphisms of V as a \emph{left} $A$-module.\\
         ${}_A V_D$     &:& $V$ is an $A$-$D$ bimodule.\\
         $A\sim B$      &:& $A$ is similar to $B$.\\
         $[A]$          &:& Similarity class of $A$ under $\sim$, or the equivalence class of finite dimensional central simple algebras similar to $A$.\\
         $\H$   &:& The quaternion algebra.\\
         $C_R(S)$       &:& The centralizer of a subset $S\subseteq R$ in ring R, which is 
                            $\{z\in R: \ zs=sz \text{ for all }s\in S\}$\\
         $A\otimes_k B$ &:& If $k$ is the base field of algebra $A$, $B$, it is often written as $A\otimes B$. If $k$ is not the base field, say $k\subseteq L$, then we write $A\otimes_L B$.\\
         $m_a(x)$         &:& The minimal polynomial of $a\in L\setminus k$ in the finite field extension $L/k$.\\
         $Gal(L/k)$     &:& The Galois group of a Galois extension $L/k$.\\
         $L^\times$     &:& The multiplicative group of a field $L$.\\
         $[V:k]$        &:& $dim_k V$, for $k$-vector space $V$.\\
         $[G:H]$        &:& If $H$ is a subgroup of a group $G$, then $[G:H]=|G/H|$ is the number of $H$-cosets of $G$.\\
         $\R_{>0}^\times$&:& The multiplicative group of positive reals. \\
         
    
    
         P.I.D          &:& The principal ideal domain.\\
         U.F.D          &:& The unique factorization domain.\\
         $m|n$          &:& The integer $m$ divides $n$.\\
         $A\defeq B$    &:& $A$ is defined to be equal to $B$.\\
         $R[X]$         &:& The ring of polynomials over a ring $R$.\\
    \end{tabularx}
}\end{table}

\chapter*{Assumptions}
\begin{itemize}
    \item All rings in this dissertation are associative and have an identity element $1$, but are not
necessarily commutative.
    \item Modules are left modules, unless stated otherwise.
    \item Let $k$ be a field, and $A$ be a $k$-algebra, unless stated otherwise.
    \item If $A\cong B$ are isomorphic rings or modules. $S\subseteq A$. Then writing $S\subseteq B$ is an abuse of notation which means the image of $S$ is in $B$ under the isomorphism $A\cong B$, or $S$ embeds into $B$.
    \item We will assume the knowledge of Galois theory in chapter $5$ of \cite{drozd2012finite} and \emph{non-commutative rings} in \cite{non-comm}. Some common results from them may be directly used without citation.
\end{itemize}

\end{romanpages}            

\chapter{The Finite Dimensional Central Simple Algebra}
In this chapter, we will introduce the finite dimensional central simple algebra over a field. They play essential role in the construction of the Brauer group in the next chapter and they will be the focus of this dissertation.
This chapter is mainly based on \cite{jacobson2009}, unless explicitly cited.

\section{Construction from Simple Algebra}

Let $k$ be a field and $A$ be a $k$-algebra.
\begin{definition}[central]\index{central algebra}
A $k$-algebra A is called \emph{central} if its center $Z(A)\cong k$.
\end{definition}

\begin{definition}[primitive]\index{primitive}\index{primitive!ring}\index{primitive!ideal}
Let $I$ be a two-sided ideal of A. Then
\begin{enumerate}[(i)]
  \item The ideal $I$ is \emph{left primitive} if
  \[ I=\emph{Ann}_A(M)=\{x\in A: xM=0\}=\bigcap_{x\in M}\emph{ann(x)} \]
  for some simple left $A$-module $M$;
   \item The ring $A$ is \emph{left primitive} if its zero ideal is left primitive, i.e., $A$ has a faithful simple left module.
  \end{enumerate}
\end{definition}

\begin{lemma} \label{1.3}
A simple ring is left primitive.
\end{lemma}

\begin{proof}
Let $A$ be a simple ring. Let $M$ be a maximal left ideal in $A$. Then $A/M$ is a simple left $A$-module. Take $I=\emph{Ann}_A(A/M)$ which is a two-sided ideal in $A$. $I\neq A$ since $A$ acts non-trivially on $A/M$. Hence $I=0$. $A$ is left primitive.
\end{proof}
 
\begin{theorem}[Artin-Wedderburn]\label{1.4}
Let A be a left primitive, left Artinian ring. Then $A\cong M_n(D)$ for some division ring $D$ and some integer $n\geq 1$.
\end{theorem}

\begin{proof}
One can find a proof in many classical texts, or \cite{non-comm}.
\end{proof}
 
\begin{lemma}
    Let $A$ be a ring. Let $M_n(A)$ is the $n\times n$ matrix with coefficients in $A$. Then
    \begin{enumerate}[(i)]
        \item The ideals of $M_n(A)$ are of the form $M_n(I)$ where $I\ideal A$, and $M_n(A)$ is simple if and only if $A$ is simple;
        \item The centers $Z(M_n(A))=Z(A)$, and $M_n(A)$ is central if and only if $A$ is central.
    \end{enumerate}
\end{lemma}
\begin{proof}
  This can be done via direct checking, which we would omit here. One can find a proof in many classical texts, or p.463 and p.473 of \cite{lam1999lectures}.
\end{proof}
 
We are now ready to construct finite dimensional central simple algebra from a finite dimensional simple algebra. Let $A$ be a finite dimensional simple algebra. Since $A$ is finite dimensional, $A$ is left Artinian. By lemma \ref{1.3}, $A$ is left primitive, so $A\cong M_n(D)$ for some division algebra $D$. Let \[ k\defeq Z(A)\cong Z(M_n(D))\cong Z(D) \]
which is a field. We can regard $A$ as a finite dimensional central simple algebra over $k$.

\section{Tensor Product of Algebras}
    \begin{definition}[tensor product of algebra]\index{tensor product of algebra}
    Let $k$ be a field. $A,B$ are $k$-algebras. The \emph{tensor product of $k$-algebra $A,B$} is the tensor product $A\otimes_k B$ of $k$-modules $A,B$ with product defined by
    \[ (a_1\otimes b_1)(a_2\otimes b_2)\defeq a_1 a_2 \otimes b_1 b_2\]
    and extends by linearity to all of $A\otimes_k B$. Denote $A\otimes_k B$ for tensor product of algebras if $A$ and $B$ are both $k$-algebras.
    
    \end{definition}
    
    $A\otimes_k B$ is a $k$-algebra with identity given by $1_A \otimes 1_B$ where $1_A$, $1_B$ are the identities of $A$ and $B$.
    
    Similar to that of modules, the tensor product of algebras is characterized by the following universal property.

    \begin{proposition}\label{univ}
      Let $A,B,C$ be a $k$-algebra. Let $e_1, e_2$ be the homomorphisms $e_1:A\to A\otimes_k B$, $x\mapsto x\otimes 1_B$ and $e_2:B\to A\otimes_k B$, $y\mapsto 1_A\otimes y$. We have
      \[ (x\otimes 1_B)(1_A\otimes y)=x\otimes y=(1_A \otimes y)(x\otimes 1_B)         \]
      i.e., 
      \begin{equation}\label{1.6.1} 
          e_1(x)e_2(y)=e_2(y)e_1(x) ,\quad   x \in A, y \in B. 
      \end{equation}
      \[ \]
      Then there is a unique algebra homomorphism (canonical homomorphism)  $f:A\otimes_k B \to C$ such that $fe_i=f_i, \quad i=1,2.$
    \end{proposition}

    \begin{center}
        \begin{tikzcd}
            A 
            \arrow{r}{e_1}
            \arrow{rd}[swap]{f_1}
            
            &A\otimes B  
            \arrow[dotted]{d}{f}
            \arrow[leftarrow]{r}{e_2}
            
            &B 
            \arrow{ld}{f_2}\\
            
            &C
        \end{tikzcd}
    \end{center}
    
    \begin{proof}
        Define $\tilde{f}:A\times B \to C$ by $(x,y)\mapsto f_1(x)f_2(y)$. This is $k$-bilinear so we have a $k$-module homomorphism $f:A\otimes B\to C$ induced by $\tilde{f}$. It remains to check that $f$ is a $k$-algebra homomorphism. Since every element of $A\otimes B$ is of the form $\sum x\otimes y$, it suffices to check it on $x\otimes y$.
        \[ f(1)=f(1_A \otimes 1_B) = f_1(1_A)f_2(1_B)=1\cdot 1=1\] and 
        \begin{align*}
          f((x_1\otimes y_1)(x_2 \otimes y_2))&=f(x_1 x_2 \otimes y_1 y_2)\\
                                              &=f_1(x_1 x_2) f_2(y_1 y_2)\\
                                              &=f_1(x_1)f_1(x_2)f_2(y_1)f_2(y_2)\\
                                              &=f_1(x_1)f_2(y_1)f_1(x_2)f_2(y_2)\\
                                              &=f(x_1\otimes y_1)f(x_2 \otimes y_2).
        \end{align*}
        In addition, we have
        \[ fe_1(x)=f(x\otimes 1_B)=f_1(x) \]
        \[ fe_2(y)=f(1_A \otimes y)=f_2(y). \]
        By relation \eqref{1.6.1}, we can see the homomorphism is unique, similar to the proof of tensor product of $k$-modules.
    \end{proof}

    \begin{proposition}\label{1.7}
        Let $A,B$ be $k$-subalgebras of a $k$-algebra $D$, then $D\cong A\otimes_k B$ if
        \begin{enumerate}[(i)]
            \item $ab=ba$ for all $a\in A, b\in B$;
            \item For a $k$-basis $\{x_\alpha\}_{\alpha \in I}$ of $A$, every $d\in D$ can be written in a unique way as $\sum_i x_{\alpha_i}b_i$ for some $\alpha_i \in I, b_i\in B.$ 
        \end{enumerate}
    \end{proposition}
        \begin{proof}
        By $(ii)$, we may define a map $\phi:D\to A\otimes_k B$, $ \sum_i x_{\alpha_i}b_i\mapsto \sum_i x_{\alpha_i}\otimes b_i$ and it linearly extends to $\sum a_ib_i\mapsto \sum a_i\otimes b_i$.
        Then
        $$\phi(ab\cdot a'b')\stackrel{(i)}{=}\phi(aa'bb')=aa'\otimes bb'=(a\otimes b)(a'\otimes b')=\phi(ab)\phi(a'b')\quad for\ a,a'\in A,\ b,b'\in B$$
        by $(i)$. Hence $\phi$ is an algebra homomorphism.
        
        $\phi$ is obviously surjective. For injectivity, note that every elements of $A\otimes B$ can be uniquely written as finite sum $\sum_j x_{\alpha_j}\otimes b_j$ for $x_{\alpha_j}\in \{x_\alpha\}_{\alpha\in I}$, $b_j\in B$. 
        
        Indeed, by the natural homomorphisms $$A\otimes B \cong(\oplus_{\alpha\in I}kx_\alpha)\otimes B\cong \oplus_{\alpha\in I}(kx_\alpha\otimes B),$$ the sum $\sum_j x_{\alpha_j}\otimes b_j=0$ forces  $x_{\alpha_j}\otimes b_j=0$ for all $j$. Then $kx_\alpha\otimes B\cong k\otimes B\cong B$ implies $x_{\alpha_j}\otimes b_j=0$ if and only if $b_j=0$, so $\sum_j x_{\alpha_j}b_j=0$. Hence $\phi$ is bijective.
    \end{proof}
    
    \begin{remark}
        If $[D:k]<\infty$, then $D=AB$ and $[D:k]=[A:k]\cdot [B:k]$ implies $(ii)$.
    \end{remark}
    \begin{proposition}\label{1.8}
    A is a $k$-algebra.
    \begin{enumerate}[(i)]
        \item $M_n(A) \cong M_n(k)\otimes_k A$;
        \item $M_m(k)\otimes M_n(k)\cong M_{mn}(k)$.
    \end{enumerate}
    \end{proposition}
    
    \begin{proof}
        $(i)$. 
         Let $\{e_{ij}\}_{1\leq i,j \leq n}$ be a basis for $M_n(k)$, where $e_{ij}$ is the matrix with $(i,j)$-entry 1, and other entries 0. Write $AI=\{aI: a \in A\}$ where 
         $aI = \left(
         \begin{smallmatrix}
         a &&&\\ &a&& \\&&\ddots &\\ &&&a 
         \end{smallmatrix} \right) $
         in $ Z(M_n(A))$. 
         Therefore $e_{ij}(aI)=(aI)e_{ij}$ for all $aI\in AI$, $e_{ij}\in M_n(k)$. $\{e_{ij}\}$ is an $A$-basis for $M_n(A)$. Every element of $M_n(A)$ can be written uniquely in the form $\sum_{k=1}^s a_k I e_{\alpha_k}$ where $e_{\alpha_k}\in \{e_{ij}\}_{1\leq i,j\leq n}$, $a_k \in A$. By \ref{1.7} we have $M_n(A) \cong M_n(k)\otimes_k A$.
         
         $(ii)$ Viewing $A=M_m(k)$ as a subalgebra of $M_{mn}(k)$ by considering the block matrix
         \[I_n(a)\defeq 
            \begin{pmatrix} 
            a &&&\\ &a&& \\&&\ddots &\\ &&&a 
            \end{pmatrix} \in M_{mn}(k)
         \] where $a\in A$ is a $m\times m$ block matrix and we have $n\times n$ such blocks with $a$ along the diagonal, $0$ elsewhere.
         
         Viewing $B=M_n(k)$ as a subalgebra of $M_{mn}(k)$ by considering the block matrix for every $b$. Define
         \[ bI_m \defeq 
         \begin{pmatrix}
              b_{11} I_m & \cdots & b_{1n} I_m \\
               \vdots & \ddots & \vdots \\
               b_{n1} I_m & \cdots & b_{nn} I_m
         \end{pmatrix},\ b=(b_{ij})_{1\leq i,j \leq n}\in B,\ b_{ij}\in k
         \]
         where $b_{ij}I_m$ is a $m\times m$ block with $b_{ij}$ along the diagonal and $0$ elsewhere.
         
         It is straightforward to check $BI_m=\{bI_m: \ b\in B\}$ is a $k$-subalgebra of $M_{mn}(k)$.
         
         By block matrix multiplication, one can verify 
         \begin{equation}\label{1.2.4.1}
             I_n(a)\cdot (bI_m)=(bI_m)\cdot I_n(a)=
         \begin{pmatrix}
              ab_{11} I_m & \cdots & ab_{1n} I_m \\
               \vdots & \ddots & \vdots \\
               ab_{n1} I_m & \cdots & ab_{nn} I_m
         \end{pmatrix}
         \end{equation}
         Let $\{c_{uv}\}_{1\leq u,v \leq mn}$ be the standard $k$-basis of $M_{mn}(k)$ with each $c_{uv}$ having $(u,v)$-entry $1$ and $0$ elsewhere. Write $u=gm+x$, $v=hm+y$. Comparing $c_{uv}$ with (\ref{1.2.4.1}) gives 
         $b_{ij}=
               \left\{
                    \begin{array}{lll}
                          1& if\ (i,j)=(g,h) &\\
                        0& otherwise& 
                    \end{array} 
                \right. 
            .$ Take $b=(b_{ij})$ be this matrix in $B$. $a$ is a matrix with $(x,y)$-entry $1$ and $0$ elsewhere. Every $c_{uv}$ corresponds to a pair $(a,b)$. The solution is unique, and exist for all basis element of $M_{mn}(k)$, hence all of $M_{mn}(k)$.
            
        By proposition \ref{1.7} we have $M_m(k)\otimes M_n(k)\cong M_{mn}(k)$.
    \end{proof}

    \begin{lemma}\label{1.14}
        Let R be a finite dimensional simple $k$-algebra. $M_1$, $M_2$ are two irreducible $R$-modules. Then $M_1\cong M_2$.
    \end{lemma}
    \begin{proof}
    Choose a minimal non-zero left ideal $I\ideal R$. $I$ is an irreducible $R$-module. Let $M$ be any irreducible $R$-module. It suffices to show $M\cong I$ as $R$-module. Regard $M$ as an $R$-representation by the map $\phi:R\to End_k(M)$ where $ker(\phi)$ is a two-sided ideal in $R$. $R$ is simple so the representation is faithful. Since $I\neq 0$, there exist $x\in M$ such that $Ix\neq 0$. Therefore the homomorphism $\psi_x:I\to M$, $r\mapsto rx$ is not a zero homomorphism. $M$ is irreducible $R$-module, so $\psi_x$ is an isomorphism. $M\cong I$ as claimed.
    \end{proof}

    \begin{theorem}\label{1.9}
        If $M_m(D_1) \cong M_n(D_2)$ for some division algebra $D_1, D_2$ and $m,n\geq 1$. Then $m=n$ and $D_1\cong D_2$.
    \end{theorem}
    \begin{proof}
    Let $A=M_m(D_1)\cong M_n(D_2)$. By semi-simplicity, $A\cong \oplus_{i=1}^m S_i\cong \oplus_{j=1}^n S'_j$, where $S_j$, $S'_j$ are  simple left $A$-modules, so all of them are isomorphic, by \ref{1.14}. We can regard $S=S_i=D_1^m$, $S'=S_j'=D_2^n$ and $A=End(S_{D_1})$. $A$ acts transitively on non-zero elements of $S$ because $D_1$ is a division algebra.
    
    Now if we show $End(_A S)\cong D_1^{op}$ we are done because $$D_1\cong D_1^{op}\cong End(_A S)\cong End(_A S') \cong D_2^{op}\cong D_2,$$ and then we have $m=n$. It remains to show 
    \begin{equation}\label{1.2.6.1}
        End(_A S)\cong D_1^{op}.
    \end{equation}
    Take $b\in End(_A S)\subseteq End(S)$, where $End(S)$ is the endomorphism ring of $S$ as an abelian group. $b$ is $A$-linear so for all $a\in A$, $$a(b\cdot s)=b\cdot (a\cdot s)\quad for\ all\ s\in S.$$ Therefore we have $ab=ba$ for all $a\in A$ in $End(S)$. $b$ is a centralizer of $A=End(S_{D_1})$ in $End(S)$.
    
    Take $b\in End(_A S)$. Fix any $0\neq s\in S$. We claim that $bs\in sD_1$. If not, $bs\notin sD_1$, then $s$ and $bs$ are $D_1$-linearly independent. There is an $D_1$-linear map $a'\in End(S_{D_1})$ sending $s$ to $0$ and $bs$ to $1$, but $$1=a'(bs)\stackrel{\eqref{1.2.6.1}}{=}b(a's)=b\cdot 0=0.$$ Now the claim holds.
    
    For this $b$, we have $bs=sd$ for some $d\in D_1$. Take another $0\neq t\in S$, then we have $bt=td'$ for some $d'$ and $s=ct$ for some $c\in A$. Then $$bs=bct=c(bt)=c(td')=(ct)d'=sd'$$ therefore $d=d'$. Therefore for every $b\in End(_A S)$, there exist $d\in D$ such that $bs=sd$ for all $s\in S$.
    
    For uniqueness of $d$, say $sd=sd'$. Take $s=(1_{D_1}, 0, 0, \cdots,0)$ gives $d=d'$.
    
    It is then obvious that $$(b_1b_2)s=b_1(b_2s)=(b_2s)d_1=(sd_2)d_1=s(d_2d_1),\quad for\ the\ unique\ d_1,d_2\in D_1.$$
    
    Therefore $End(_A S)\cong D_1^{op}$ and this concludes the whole proof.

    \end{proof}

\section{The Enveloping Algebra}

    Let $V$ be a left $A$-module. $D=End({}_{A}V)$. We can write left $A$-module endomorphisms of $V$ on the right and define the composition of such endomorphisms by the rule
    \[\gamma(\alpha \beta)=(\gamma\alpha) \beta,\ \gamma\in V,\ \alpha,\beta \in D. \]
    We can regard $V$ as a $A$-$D$-bimodule, written ${}_A V_{D}$ and the two actions are compatible, i.e., 
    \[a(v\alpha)=(av)\alpha,\ a\in A,\  \alpha \in D.\]
    In particular, if we regard $A$ as a left-$A$ module, then the action of $End({}_A A)$ can be regarded as right multiplication by an element of $A^{op}$ via $a\mapsto (a)\alpha=a\cdot(1)\alpha$ where $(1)\alpha\in A^{op}$.
    
    \begin{definition}[enveloping algebra]\label{1.3.1}\index{enveloping algebra}
        The \emph{enveloping algebra} $A^e$ of a $k$-algebra $A$ is \[A^e\defeq A\otimes_k A^{op}.\] If A is a $k$-subalgebra of B, the natural action of $A^e$ on $B$ (hence in particular A) is given by 
        \[ (\sum a_i \otimes b_i)x=\sum a_i x b_i,\ a_i\in A,\ b_i\in A^{op},\ x\in B. \]
    \end{definition}
    
    By the universal property of tensor product there is a homomorphism from $A^e$ to $B$ so this is well-defined and it is a module action by direct verification. Observe that an $A^e$-submodule of $A$ is a two-sided ideal of $A$ and if $A$ is simple, $A$ is $A^e$-irreducible. To understand $End_{A^e}(A)$, note that any $\alpha \in End_{A^e}(A)$ is both left $A$-linear and right $A^{op}$-linear. A left $A$-linear (right $A^{op}$-linear) map is in a natural way the same as right $A^{op}$ (left $A$) multiplications. Hence the action of $\alpha$ is $x\mapsto cx=xd$, for $c\in A$, $d\in A^{op}$. Putting $x=1_A$ gives $c=d$. Hence $c\in Z(A)$. If $A$ is central $k$-algebra then $c\in k$ and the action is $x\mapsto cx$, for $c\in k$, i.e., $End_{A^e}(A)\cong k$.
    
    \begin{lemma}[Schur's lemma]
        Let V be a simple left $A$-module. Then $D=End({}_A V)$ is a division ring.
    \end{lemma}
    \begin{proof}
        One can find a proof in many classical texts, or \cite{non-comm}.
    \end{proof}
    
    \begin{theorem}[Jacobson's density]\label{1.12}
        Let $A$ be a ring. Let $V$ be a simple left $A$-module. Let $X\subseteq V$ be a finite $D$-linearly independent subset of $V$ where $D=End_A(V)$ (by Schur's lemma $D$ is a division ring). Then for every $\alpha\in End_D(V_D)$ there exist $a\in A$ such that $\alpha (x)=ax$ for all $x\in X$. 
    \end{theorem}
    
    \begin{remark}
        This is a density theorem in the following sense: $V$ is a vector space over division ring $D$. $A$ is said to act densely on $V$ (or $A$ is said to be dense on $V$) if for every $n$ and $x_1,\cdots,x_n$ in $V$, which are linearly independent over $D$, and corresponding $n$ elements $y_1,\cdots,y_n$ in $V$, there is an element $a\in A$ such that $y_i=ax_i$ for all $i$. This theorem says that $A$ acts densely on $V$.
    \end{remark}
    \begin{proof}
         One can find a proof in many classical texts, or \cite{non-comm}.
    \end{proof}

\section{Some Main Results}
    We are now ready to prove some main results for finite dimensional central simple algebra.
    
    \begin{theorem}\label{1.13}
        If $A$ is a finite dimensional central simple algebra over a field $k$. Then
        \[ A^e=A\otimes_k A^{op}\cong M_n(k) ,\  n=[A:k].\]
    \end{theorem}
    \begin{proof}
        Regard $A$ as a $A^e$-module. $A$ is simple so $A$ is $A^{e}$-irreducible and $End_{A^e}(A)\cong k$ as above. $A$ is finite dimensional over $k$ so choose a $k$-basis $\{x_1,\cdots,x_n\}$. By Density Theorem \ref{1.12}, we have a epimorphism from $A^e$ to $End_k(A)$. Since \[[A^e:k]=n^2=[End_k(A):k],\] we have $A^e\cong End_k(A)\cong M_n(k)$.
    \end{proof}

    \begin{theorem}\label{1.4.2}
    Let $A$ be a finite dimensional central simple subalgebra of an algebra $B$ over $k$. Then
    \begin{enumerate}[(i)]
        \item $B\cong A\otimes_k C$, $C=C_B(A)$;
        \item There is a bijective correspondence between ideals $I\triangleleft C$ and ideals in $B$ via $I\mapsto A\otimes I$;
        \item $Z(C)=Z(B)$.
    \end{enumerate}
    \end{theorem}
    
    \begin{proof}
        We can regard $B$ as $A^e$-module. By theorem \ref{1.13}, $A^e\cong M_n(k)$ is a simple ring since $k$ is a field. $A$ is simple so $A$ is $A^e$-irreducible. By lemma \ref{1.14} above, we know that all irreducible $A^e$-modules are isomorphic to $A$. Therefore $B$ is a direct sum of all irreducible $A^e$-modules which are isomorphic to $A$.
        
        $(i)$.
            We want to apply \ref{1.7}. By definition of $C$, we have $ac=ca$ for all $a\in A,\ c\in C$. Observe that $1$ is a generator of $A$ as $A^e$-module having the properties
            \[ (a\otimes 1)\cdot 1=a\cdot 1=(1\otimes a)\cdot 1,\] \[(a\otimes 1)\cdot 1=0\ \textrm{implies}\ a=0.\]
            Since any irreducible $A^e$-module is isomorphic to $A$, we can find $c_\alpha$ in every irreducible $A^e$-module with the properties:
            \begin{itemize}
                \item $(a\otimes 1)\cdot c_\alpha=(1\otimes a)\cdot c_\alpha$;
                \item $(a\otimes 1)\cdot c_\alpha=0\ \textrm{implies}\ a=0$;
                \item $c_\alpha$ is an $A^e$-generator, hence an $A$-generator by the first property.
            \end{itemize}
            Therefore we can write $B=\oplus_\alpha A c_\alpha$ with $c_\alpha a=ac_\alpha$, and $ac_\alpha=0$ implies $a=0$ for all $a$. Therefore $c_\alpha\in C$ and every element of $B$ can be written in a unique way of $\sum a_\alpha c_\alpha$, $a_\alpha \in A$. It remains to show $\{c_\alpha\}_\alpha$ is a basis of $C$.
            
            Take $c\in C\subset B$ and write $c=\sum a_\alpha c_\alpha$. For any $a\in A$, $ac=ca$ implies $aa_\alpha=a_\alpha a$. Hence $a_\alpha \in Z(A)\cong k$. Therefore $c\in \sum kc_\alpha$ and $\{c_\alpha\}_\alpha$ is a basis of $C$.
            
            Now applying \ref{1.7} gives the result (i).
            
            $(ii)$.
            Write $a\in A$ for $a\otimes 1_C$, and $c\in C$ for $1_A\otimes C$ in $B$. Note that $ac=ca$ so $B=AC=CA$. $I$ is an ideal in $C$. $AI$ is an ideal in $B$. By the universal property of tensor product \ref{univ} there is a unique homomorphism $A\otimes_k C\to AC$. By $(i)$, this is an isomorphism. Let $\{x_1=1,x_2,\cdots,x_n\}$ be a basis of $A$. By the isomorphism we see that every element of $B$ is written uniquely in the form $\sum x_i c_i$, $c_i\in C$ so elements of $AI$ is of the form $\sum_{i=1}^n x_i d_i$ for $d_i\in I$. 
            
            Claim that $AI\cap C=I$. Indeed, let $y\in AI\cap C$. $y=\sum_{i=1}^n x_i d_i$ for some $d_i\in I$. Since $y\in C$, $y=c_1=x_1 c_1$ for some $c_1\in C$. Since $y\in B$ the uniqueness gives $\sum_{i=1}^n x_i d_i=x_1 c_1$. Since $\{x_i\}_{i=1}^n$ form a basis, $y=x_1 d_1$ for $d_1\in I$. Because $x_1=1$, $y\in I$. $AI\cap C\subseteq I$. But then $I\subseteq AI\cap C$ gives $I=AI\cap C$. For injectivity, set $AI_1=AI_2$, and then $I_1=AI_1\cap C=AI_2\cap C=I_2$.
            
            For surjectivity, consider any ideal $I'$ in $B$. We need to show $I'$ is of the form $AI$ where $I\defeq I'\cap C$, an ideal in $C$. We claim that if $\sum_{i=1}^n x_i c_i\in I'$ then $c_i\in I'$ for all $i$. To see it consider a $k$-linear map $T_j:A\to A$ sending $x_j$ to $1$ and other $x_i$, $i\neq j$, to $0$. Since $End_{A^e} A\cong k$ by the comments following \ref{1.3.1}, density theorem \ref{1.12} gives that there exist $w_j\in A^e$ acting on $A$ in the same way as $T_j$. $w_j\in A\otimes_k A^{op}$ so $w_j=\sum_{l=1}^n x_l\otimes v_{jl}$ for $v_{jl}\in A$, and $w_j I'\subseteq I'$.
            
            Then we have
            \begin{align*}
                w_j\cdot (\sum_{i=1}^n x_i c_i) &=\sum_{l=1}^n(x_l(\sum_{i=1}^n x_i c_i) v_{jl})=
            \sum_{l=1}^n \sum_{i=1}^n (x_l x_i v_{jl})c_i\\&=
            \sum_{i=1}^n(\sum_{l=1}^n x_l x_i v_{jl})c_i=\sum_{i=1}^n(w_j\cdot x_i)c_i\\&=
            \sum_{i=1}^n T_j(x_i)c_i=c_j\in w_j I'\subseteq I'.
            \end{align*}
            This is true for all $j=1,\cdots, n$. Moreover $c_j\in C$ so $c_j\in I'\cap C=I$. $\sum_{i=1}^n x_i c_i$ is therefore in $AI$, and hence every ideal in $B$ is of the form $AI$, $I$ is an ideal of $C$.
            
            $(iii)$. $C$ is a subalgebra of $B$ so $Z(B)\subseteq Z(C)$. On the other hand, if $c\in Z(C)$, $c$ commutes with every $a\in A$ by definition of C. Hence $c$ commutes with every element of $AC=B$. Therefore $c\in Z(B)$. $Z(C)\subseteq Z(B)$. We have $Z(C)=Z(B)$.
    \end{proof}
    
    \begin{corollary}
    A is finite dimensional central simple algebra over $k$. $C$ is an arbitrary $k$-algebra. Let $B\cong A\otimes_k C$. Then
    \begin{enumerate}[(i)]
        \item There is a bijective correspondence between ideals $I\triangleleft C$ and ideals in $B$ via $I\mapsto A\otimes_k I$;
        \item $Z(C)=Z(B)$.
    \end{enumerate}
    \end{corollary} 
        
    \begin{proof}
    Write $a\in A$ for $a\otimes 1_C$ and $c\in C$ for $1_A\otimes c$. Then $ac=ca$ for all $a\in A,\ c\in C$. Let $Z$ be the centralizer of $A$ in $B$. We have $C\subseteq Z$. Claim that $C=Z$.
    
    \noindent
    To prove the claim, consider a basis $\{y_\beta\}_\beta$ of C. Then element in $B$ can be written in a unique way of the form $b=\sum a_\beta y_\beta$ for $a_\beta\in A$. If $b=\sum a_\beta y_\beta \in Z$, then $ab=ba$ if and only if $a_\beta a=aa_\beta$ for all $a_\beta$, $a\in A$. Therefore $a_\beta\in Z(A)\cong k$. $b=\sum a_\beta y_\beta$ for $y_\beta\in k$. Then $b\in C$ and $Z\subseteq C$.
    
    \noindent
    Now $C=Z$ and the result follows from the previous theorem.
    \end{proof}

    \begin{corollary}\label{1.17}
    Let $A$ be a finite dimensional central simple algebra over $k$. $C$ is an arbitrary $k$-algebra. Then
    \begin{enumerate}[(i)]
        \item $A\otimes_k C$ is simple if $C$ is simple;
        \item $A\otimes_k C$ is central if $C$ is central.
    \end{enumerate}
    That is, if $A,\ C$ are finite dimensional central simple $k$-algebras, then $A\otimes_k C$ is a finite dimensional central simple $k$-algebra.
    \end{corollary}
    
    \begin{proof}
    This follows immediately from the last corollary.
    \end{proof}

\chapter{Definition of the Brauer Group}
\numberwithin{theorem}{chapter}
    The properties exhibited in the previous chapter result in an important group, the Brauer Group. This chapter is mainly based on \cite{jacobson2009}, unless explicitly cited. Throughout this chapter, let $A,B,C$ be finite dimensional central simple algebra over a field $k$. Let $D,\ D_1,\ D_2$ be finite dimensional central division algebra over $k$. 
    
    \begin{definition}[similar]\index{similar}
        We say $A$ is \emph{similar} to $B$, written $A\sim B$ if there exist positive integers $m,n$ such that $M_m(A)\cong M_n(B)$, or equivalently, $M_m(k)\otimes_k A\cong M_n(k)\otimes B$.
        
        We write $[A]$ for similarity class of A under $\sim$, that is, the equivalent class (see below) of finite dimensional central simple $k$-algebras similar to $A$.
    \end{definition}

    \begin{lemma}
    The relation $\sim$ is an equivalence relation.
    \end{lemma}
    \begin{proof}
        Reflexive and symmetric property is obvious. Now we show transitivity.\\
        If $A\sim B$ and $B\sim C$ then there exist $m,n,s,t$ such that $M_m(A)\cong M_n(B), M_s(B)\cong M_t(C)$. By proposition \ref{1.8}, commutativity and associativity of the tensor multiplication (under isomorphism), we have
        \[M_{ms}(A)\cong M_{ms}(k)\otimes_k A\cong M_s(k)\otimes_k M_m(k)\otimes_k A\cong M_s(k)\otimes_k M_m(A)\]
        \[\cong M_s(k)\otimes_k M_n(B)\cong M_s(k)\otimes_k M_n(k) \otimes_k B\cong M_s(B)\otimes_k M_n(k) \]
        \[\cong M_t(C) \otimes_k M_n(k)\cong M_t(k)\otimes_k M_n(k) \otimes_k C\cong M_{tn}(k)\otimes_k C\cong M_{tn}(C).\]
        Thus $A\sim C$.
    \end{proof}

    \begin{lemma} \label{2.3}
        Let $A\cong M_n(D_1)$, $B\cong M_m(D_2)$. $A\sim B$ if and only if $D_1\cong D_2$.
    \end{lemma}
    \begin{proof}
    By Artin-Weddernburn Theorem \ref{1.4}, we see $A\cong M_n(D)$ for some unique division ring $D$ and integer $n\geq 1$, where the uniqueness follows from theorem \ref{1.9}. We have \[Z(D)\cong Z(M_n(D))\cong Z(A)\cong k.\] Since $A$ is finite dimensional, $D$ is finite dimensional. Hence every finite dimensional central simple algebra correspond to a finite dimensional central division algebra. Conversely, for a $D$ as indicated above, $M_n(D)$ is finite dimensional central simple. Therefore if $A\cong M_n(D_1)$, $B\cong M_m(D_2)$, then $A\sim B$ if and only if $D_1\cong D_2$.
    \end{proof}

    \begin{lemma}\label{2.4}
    $[A]=[k]$ if and only if $A\cong M_n(k)$ for some $n$.
    \end{lemma}
    \begin{proof}
        (if) If $A\cong M_n(k)$, $M_m(A)\cong M_m(k)\otimes_k M_n(k)\cong M_{mn}(k)$. Hence $[A]=[k]$.
        
        (only if) By Artin-Weddernburn Theorem \ref{1.4}, $A\cong M_l(D)$ for some division algebra $D$. $[A]=[k]$ implies $M_m(k)\otimes A\cong M_s(k)$ so \[M_s(k)\cong M_m(k)\otimes_k M_l(D)\cong M_m(k)\otimes_k M_l(k)\otimes_k D\cong M_{ml}(k)\otimes_k D\cong M_{ml}(D).\]
        Then by theorem \ref{1.9} we have $s=ml$ and $k\cong D$. Therefore $A\cong M_l(k)$.
    \end{proof}
    
    \begin{definition}[the Brauer group]\index{the Brauer group}
        Let $Br(k)$ denote the set of similarity classes of finite dimensional central simple $k$-algebras. Endow a binary operation on $Br(k)$ by
        \[ [A][B]\defeq [A\otimes_k B]. \] $Br(k)$ is the \emph{Brauer Group} over $k$.
    \end{definition}
    
    \begin{theorem}\label{2.6}
        $Br(k)$ is a well-defined abelian group.
    \end{theorem}
    \begin{proof}
        For well-defineness of the binary operation, note that $[A\otimes_k B]\in Br(k)$ by corollary \ref{1.17}. If $A\sim A'$, $B\sim B'$, then it follows from \ref{1.8} there exist $m_1,m_2,n_1,n_2$ such that 
        \[ M_{m_1}(k)\otimes_k A\cong M_{m_2}(k)\otimes_k A'\quad M_{n_1}(k)\otimes_k B\cong M_{n_2}(k)\otimes_k B'
        \]
        This implies 
        \[M_{m_1 n_1}(k)\otimes_k A\otimes_k B\cong
        M_{m_2 n_2}(k)\otimes_k A'\otimes_k B'\]
        so we have $A\otimes_k B\sim A'\otimes_k B'$.
        
        $[k]$ is the unique identity element. The uniqueness follows from lemma \ref{2.4}. It is the identity because for every $[A]$,
        \[ [A][k]\cong M_m(D)\otimes_k M_n(k)\cong M_{mn}(D)=[A]\] for some unique  $m,n\geq 1$ and division algebra $D$ by lemma \ref{2.3} and proposition \ref{1.8}.
        
        Commutativity and associativity of tensor algebra induces commutativity and associativity of $Br(k)$.
        
        By theorem \ref{1.13}, we have $A\otimes_k A^{op}\cong M_n(k)$ for some $n$ so $[A^{op}]=[A]^{-1}$.
        
    \end{proof}
    
    The first example of the Brauer group is when $k$ is algebraically closed field.
    
    \begin{lemma}\label{2.7}
        The only finite dimensional division algebra over an algebraically closed field $k$ is $k$ itself.
    \end{lemma}
    
    \begin{proof}
    Suppose $D$ is a division algebra with $[D:k]=n$. Take any $x\in D$. $1,x,x^2,\dots, x^n$ are linear dependent. There exist a monic polynomial of least degree $f\in k[X]$ with $f(x)=0$. Since $k$ is algebraically closed, $f(x)=(x-a)g(x)$ for some $g(x)\neq0$ by the least degree property. Since $D$ is a division algebra, $f(x)=0$, $g(x)\neq 0$ implies $x-a=0$. Now $a\in k$, $x\in k$. Therefore $D=k$.
    \end{proof}
    
    \begin{corollary}
      If $k$ is an algebraically closed field, $Br(k)\cong 1$.
    \end{corollary}
    \begin{proof}
    The result is immediate by lemma \ref{2.3} and lemma \ref{2.7}
    \end{proof}

    \chapter{The Brauer Group of Real Numbers}
    \numberwithin{theorem}{section}
    In this chapter we want to identify the group $Br(\R)$. It is mainly based on \cite{drozd2012finite}, unless explicitly cited. We need to prove the famous theorem due to Frobenius on real division algebras. First, we need some tools.

    \section{Maximal Subfield}
    
    \begin{definition}[maximal subfield]\index{maximal subfield}
    A \emph{maximal subfield} of a ring is the fields which are not contained in any larger subfield.
    \end{definition}

    \begin{theorem}\label{3.2}
    Let $A$ be a finite dimensional central simple algebra. $B$ is its simple subalgebra. $B'=C_A(B)$. Then
    \begin{enumerate}[(i)]
        \item $B'$ is simple;
        \item $C_A(B')=B$;
        \item $[A:k]=[B:k][B':k]$;
        \item If $B'\cong M_m(D)$, then $A\otimes B'\cong M_n(D)$ and $m|n$.
    \end{enumerate}
    \end{theorem}
    
    \begin{proof}
    Consider $A_f$ as a $A$-$B$-bimodule, where the action is $a\cdot x\cdot b=axf(b)$, $f$ is the inclusion map from $B$ into $A$. Then $A_f$  is an $A\otimes_k B^{op}$ module. $A\otimes_k B^{op}$ is simple (\ref{1.17}), so $A\otimes_k B^{op}\cong M_n(D)$ is a semisimple algebra (c.f. proposition 2.3.3 \cite{drozd2012finite}). Let $A_f=U^m$ where $U$ is the simple $A\otimes_k B^{op}$-module (\ref{1.14}) and $U\cong D^n$ (c.f. proposition 2.3.4 \cite{drozd2012finite}). In the proof of  \ref{1.9}, we see $E_{A\otimes_k B^{op}}(U)\cong D^{op}\cong D$. Hence, $$E_{A\otimes_k B^{op}}(A_f) \cong M_m(E_{A\otimes_k B^{op}}(U))\stackrel{\ref{1.9}}{\cong} M_m(D)$$ by lemma 2.19 in \cite{non-comm}.
    
    Take $\phi\in E_{A\otimes_k B^{op}}(A_f)$. $\phi(x)=x\phi(1)=xa_0$ for $a_0:=\phi(1)\in A$. For all $b\in B$, since $\phi$ is right $B$-linear, $b\cdot a_0=b\cdot \phi(1)=\phi(b)=\phi(1)\cdot b=a_0\cdot b$, so $a_0\in C_A(B)=B'$.
    
    Conversely, for any $a_0\in B'$, the right multiplication by $a_0$ is always in $E_{A\otimes_k B^{op}}(A_f)$. Hence $B'\cong E_{A\otimes_k B^{op}}(A_f)\cong M_m(D)$, so $B'$ is simple. (i) is proved.
    
    Now let $d=[D:k]$. Then $[A:k]=m[U:k]=mnd$, $$[A:k][B:k]=[A\otimes B^{op}:k]=[M_n(D):k]=n^2d,\ \ [B':k]=m^2d,$$ so $[B:k]=\frac{n^2d}{mnd}=\frac{n}{m}$, giving $m|n$ and $[A:k]=[B:k][B':k]$. (iii) (vi) are proved.
    
    Let $B''=C_A(B')$. $B\subseteq C_A(C_A(B))$ for sure. Now by symmetry of the argument, $B'$ is simple so $[A:k]=[B'':k][B':k]$. This means $[B'':k]=[B:k]$. This forces $B=B''$. (ii) is proved.
    \end{proof}
    
    \begin{lemma}\label{3.3}
        Let $A$ be a finite dimensional central simple algebra. $L$ be its subfield. Then the following statements are equivalent:
        \begin{enumerate}[(i)]
            \item $L=C_A(L)$;
            \item $[A:k]=[L:k]^2$.
        \end{enumerate}
    \end{lemma}
    
    \begin{proof}
        Observe $L\subseteq C_A(L)$ for any subfield $L$ of $A$. By \ref{3.2} (iii), we know
        $$[A:k]=[L:k][C_A(L):k]\geq [L:k]^2.$$
        The equality holds if and only if $[C_A(L):k]=[L:k]$ if and only if $C_A(L)=L$.
    \end{proof}
    
    \begin{definition}[strictly maximal]\index{strictly maximal subfields}
        A subfield L of a finite dimensional central simple algebra is \emph{strictly maximal} if it satisfies the equivalent conditions in the last lemma \ref{3.3}.
    \end{definition}
    
    \begin{theorem}\label{3.1.5}
    If $D$ is a finite dimensional division algebra, then
    \begin{enumerate}[(i)]
        \item every maximal subfield is strictly maximal;
        \item if in addition $D$ is central, $D\otimes L\cong M_n(L)$ where $L$ is a maximal subfield and $n=[L:k]$.
    \end{enumerate}
    \end{theorem}
    
    \begin{proof}
        If $a\in C_D(L)\setminus L$, take $f(x)\in L[X]$ to be a polynomial and $f(a)$ is a commutative subalgebra of $D$. Any commutative subalgebra of a division algebra is a field (c.f. 1.2.3 \cite{drozd2012finite}), so we have $L\subsetneq f(a)\subseteq D$. The maximality of $L$ implies $L=C_D(L)$. By \ref{3.2}, $L\cong C_A(L)\cong End_{D\otimes L}(D)\cong M_m(D)$. Therefore $m=1$, $D\cong L$ (\ref{1.14}). By $(iv)$ of \ref{3.2} the result follows with a dimension argument showing $n=[L:k]$.
    \end{proof}

    \section{The Frobenius Theorem}
    
    \begin{theorem}[Skolem-Noether]\label{3.2.1}
    If $f$ and $g$ are two homomorphisms of a finite dimensional simple algebra $B$ into a finite dimensional central simple algebra $A$, then there is an invertable element $a\in A$ such that $g(x)=af(x)a^{-1}$ for all $x\in B$. Equivalently, any homomorphism from $A$ to $B$ can be extended into an inner automorphism of $B$.
    
    \end{theorem}
    \begin{proof}
    Denote two copies of $A$, $A_f$, $A_g$ as $A\otimes_k B^{op}$ modules via the action $$\sum (a\otimes b)\cdot x=\sum a\cdot x\cdot f(b)\quad \sum (a\otimes b)\cdot y=\sum a\cdot y\cdot g(b),$$ for $a\in A$, $b\in B^{op}$, $x\in {A_f}$, $y\in {A_g}$ respectively. By corollary \ref{1.17}, $A\otimes_k B^{op}$ is simple, finite dimensional. Hence all its modules are direct sum of the unique simple module (\ref{1.14}). Since $[A_f:k]=[A_g:k]$, $A_f\cong{A_g}$ as $A\otimes_k B^{op}$ modules.
    
    Let $\phi:{A_f}\cong{A_g}$ be this $A\otimes_k B^{op}$-module isomorphism. Considering ${A_f}$ and ${A_g}$ as left $A$-regular module, we have $\phi(x)=x\cdot \alpha$ for fixed invertable $\alpha=\phi(1_A)\in A$
    
    Considering ${A_f}$, ${A_g}$ as right $B^{op}$-modules, we have
    \[      \phi(x\cdot f(b))=\phi(x)\cdot g(b) \ for\ all\  x\in {A_f} \text{, } b\in B       \]
    Taking $x=1_A$ gives $\phi(f(b))=\phi(1_A)g(b)$. This is $f(b)\alpha=\alpha g(b)$ for all $b\in B$. Hence $f(b)=\alpha g(b)\alpha^{-1}$, $\alpha$ is an invertable element in $A$.
    \end{proof}
    
    \begin{corollary}
    Every automorphism of a finite dimensional central simple algebra is inner. In particular, every automorphism of the algebra $M_n(k)$ is inner.
    \end{corollary}
    \begin{proof}
        Take $B=A$ and $g=Id$, the identity map, will give the result.
    \end{proof}
    
    Now we are ready to prove the famous
    
    \begin{theorem}[Frobenius Theorem]\label{3.8}
        The only finite dimensional division algebra over $\R$ are $\R$, $\C$, and $\H$.
    \end{theorem}
    \begin{proof}
        Let $D$ be a finite dimensional division algebra over $\R$. Take $a\in D$ and $m_a(x)$ be the minimal polynomial of the element $a$ over $\R$. Since $D$ is a division algebra, $m_a(x)$ is irreducible by minimality. By fundamental theorem of algebra, $m_a(x)$ is of degree one or two. 
        
        If $m_a(x)$ is of degree one, the $a\in \R$ and we get $D\cong \R$. 
        
        If $m_a(x)$ is of degree two, let $a=u+vi$, $u,v\in \R$, $v\neq 0$. Then $a^2=u^2-v^2+2uvi$ and $a$ satisfy $x^2-2ux+u^2+v^2=0$. Hence $a$ satisfy a quadratic of the form $x^2+2px+q$ where $p^2< q$. Now completing the square $(x+p)^2+q-p^2$. We can see $\frac{x+p}{\sqrt{q-p^2}}$ is a root of the polynomial $x^2+1$. Therefore, the subfield $\R[a]\cong \C$. $\C$ is algebraically closed, so $D\cong \C$ in this case.
        
        Thus the finite field extensions of $\R$ are $\R$, $\C$. 
        
        To explore other cases, consider $D\neq \R$, $\C$. Let $L$ be the maximal subfield of $D$. $L\neq \R$ otherwise $D=\R$ by \ref{3.1.5}. Hence $L=\C$, $[L:\R]=2$. By \ref{3.1.5} again $[D:\R]=4$. The center of $D$ is $Z(D)=\R$ because the only finite dimensional division algebra over $\C$ is $\C$ by \ref{2.7}.
        
        Let $i$ be the element in $L$ with $i^2=-1$. The complex conjugation map is an automorphism of $L$ in $D$ which sends $i\mapsto -i$. Therefore by Skolem-Noether theorem \ref{3.2.1}, considering the embeddings of $L$ into $D$, there exist an invertable $j\in D$ such that $-i=jij^{-1}$, i.e., $ji=-ij$.
        
        Since $j$ does not commute with $i$, $j\notin L$. $\{1,i,j\}$ are linearly independent. Observe that $j^2 i=-jij=ij^2$ so $j^2\in C_D(L)=L$ by \ref{3.2}. Thus $j^2=\alpha+\beta i$, for $\alpha,\beta\in \R$. But $j^2$ commutes with $j$ so $j(\alpha+\beta i)=(\alpha+\beta i)j$. This implies $\beta=0$. $j^2=\alpha\in \R$. If $\alpha>0$, $(j-\sqrt{a})(j+\sqrt{a})=0$ and $j\in \R\subseteq L$. This is impossible so $\alpha<0$. Replacing $j$ by $j/\sqrt{-\alpha}$ we have $j^2=-1$.
        
        Now let $k=ij$. We have $k^2=ijij=-1$, $ik=i^2j=-j$ and $ki=iji=-i^2j=j$. Similarly, $jk=-kj=i$. Therefore $i,j,k$ satisfy all the relations of the canonical basis of $\H$. There is a homomorphism $\psi: \H\to D$. Since $\H$ is a division algebra, $\psi$ is injective. $[\H:\R]=[D:\R]$ implies $\psi$ is an isomorphism. This concludes the proof.
    \end{proof}
    
    \begin{lemma}\label{3.9}
    $\H\otimes_k \H \cong M_4(\R)$.
    \end{lemma}
    \begin{proof}
        Construct a map $\phi:\H\to\H^{op}$, $a+bi+cj+dk\mapsto a-bi-cj-dk$, for $a,b,c,d\in \R$. It is easy to check that this is an $\R$-algebra isomorphism. Thus $\H\cong \H^{op}$ and by theorem \ref{1.13}, $\H\otimes_k \H \cong M_4(\R)$.
    \end{proof}
    
    \begin{corollary}
    $Br(\R)\cong C_2$.
    \end{corollary}
    \begin{proof}
      Since the only 2 non-isomorphic finite dimensional central division algebra over $\R$ is $\R$ and $\H$, by Frobenius theorem \ref{3.8}. $[\H]^2=[\R]$ by lemma \ref{3.9}. $Br(\R)$ is generated by $[\H]$ so it is isomorphic to $C_2$.
    \end{proof}

\chapter{The Cyclic Algebra}

    This chapter introduces an important type of central simple algebra, the cyclic algebra, as a special case of the crossed product. We will see from the later chapters that all central simple algebras over a local field are of this form.
    
    The content of this chapter is mainly based on \cite{drozd2012finite}, unless cited otherwise.
    
\section{The Crossed Product}
    For terminologies and classical results of Galois Theory, the reader is referred to chapter 5 of \cite{drozd2012finite}.

    \begin{lemma}[Noether]\label{4.1.1}
       Let $D$ be a finite dimensional central division algebra over $k$. Then there is a maximal subfield $L\in D$ which is separable over $k$.
       
    \end{lemma}
    
    \begin{proof}
    If $char(k)=0$, then $k$ is perfect so every subfield of $D$ is separable. We can assume $char(k)=p>0$.
    We shall construct an element in $D\setminus k$ which is separable over $k$.
    Take any $a\in D\setminus k$. Let $f(x)=m_a(x)\in k[X]$. If $a$ is not separable, then $f(x)$ has a multiple root $\alpha$ (a root with multiplicity greater than $1$). Then $gcd(f,f')\neq 1$ as they have common factor $(x-\alpha)$ in the splitting field. $f$ is irreducible and $deg f'<deg f$. This forces $f'(x)=0$. Thus, $f(x)$ is a polynomial in $x^p$. Now let $f(x)=g(x^p)$, for some $g(x)\in k[X]$. If $g(x)$ is separable then we find $a^p$ as root of $g(x)$ which is separable, otherwise continue the above process and finally we can find $a^{p^n}$ is separable but not $a^{p^{n-1}}$, for some $n\geq 1$. Denote $b=a^{p^{n-1}}$, so $b^p$ is separable.
    
    If $b^p\notin k$, we are done, otherwise, consider the map $\delta:D \to D$, $d\mapsto db-bd$. Since $b \notin k$, $b^p\in k=Z(D)$, there exist $d_0\in D$ such that $\delta(d_0)\neq 0$, $\delta^p(d_0)=d_0b^p-b^pd_0=0$.
    
    Let $m$ be the least number such that $\delta^m(d_0)=0$. Let $t:=\delta^{m-1}(d_0)$. $w=\delta^{m-2}(d_0)$. $u=b^{-1}t$. Then $t=\delta(w)=wb-bw$. Since $\delta(t)=0$, $ub=bu$. Now $$b=tu^{-1}=(wb-bw)u^{-1}=wbu^{-1}-bwu^{-1}=wu^{-1}b-bwu^{-1}=\delta(wu^{-1})$$
    
    Let $c=wu^{-1}$, $b=cb-bc$ so $c=1+bcb^{-1}$. By the same argument above for $a$ we can show $c^{p^n}$ is separable for some $n\geq 0$. $c^{p^n}=1+bc^{p^n}b^{-1}$ so $c^{p^n}$ does not commute with $b$. $c^{p^n}\notin Z(D)=k$. We are done with finding a separable element over $k$ but not in $k$.
    
    Now we prove the lemma by induction on $[D:k]$. For $[D:k]=1$, this is trivial. Suppose it is true for all division algebra with less dimensions over $k$.
    
    Pick $a\in D\setminus k$ separable over $k$. Let $F=k(a)$. $D_1=C_D(F)$. Then $F=C_D(D_1)$, $[D:k]=[D_1:k][F:k]$ (theorem \ref{3.2}). As a result, $F=Z(D_1)$ because $F\subseteq D_1$.
    
    Since $[D_1:k]<[D:k]$, there is a maximal subfield $L$ of $D_1$ which is separable over $F$ and $[D_1:F]=[L:F]^2$(lemma \ref{3.3}). Now $$[D:k]=[D_1:k][F:k]=([D_1:F][F:k])[F:k]=[L:F]^2[F:k]^2=[L:k]^2.$$ Thus $L$ is a maximal subfield of $D$ (lemma \ref{3.3}). Moreover, $F$ is separable over $k$ and $L$ is separable over $F$, so $L$ is separable over $k$ (c.f. 5.3.7 of \cite{drozd2012finite}). This concludes the proof.
        
    \end{proof}
    
    \begin{definition}[splitting field]\index{splitting field}
        A field $L$ is call a \emph{splitting field} of a central simple algebra $A$ if $A\otimes_k L\cong M_n(L)$ for some $n$.
    \end{definition}
    
    Let $L/k$ be a finite field extension. For every finite dimensional central division algebra $D$,  $D_L\defeq D\otimes_k L$ is central simple over $L$. $D_L\cong M_n(D')$ where $D'$ is a central simple division algebra over $L$. Then $\phi:Br(k)\to Br(L)$, $D\mapsto D'$ is a group homomorphism by direct verification with $ker(\phi)=\{D: D\otimes L\cong M_n(L)\text{ for some } n\}$. This is the same as saying $L$ is a splitting field of $D$. This subgroup is denoted as in the following definition.
    
    \begin{definition}[$Br(L/k)$]\index{$Br(L/k)$}
    $Br(L/k)\defeq \{D\in Br(k):L \text{ is a splitting field of } D \}$
    \end{definition}
    
    \begin{corollary}\label{4.1.4}
    Every finite dimensional central simple algebra has a Galois splitting field. In other words, $Br(k)=\bigcup_L Br(L/k)$ where $L$ is a Galois extension of $k$.
    \end{corollary}
    
    \begin{proof}
    Let $A$ be the finite dimensional central simple algebra. $A\cong M_n(D)$. By \ref{4.1.1}, there is a separable $L\subseteq D$ and $D\otimes L\cong M_m(L)$. From classical Galois theory we know there is a Galois extension $L'$ of $k$ containing $L$(c.f. 5.4.5 in \cite{drozd2012finite}). $L'$ is a splitting field because $$A\otimes L'\cong A\otimes L\otimes L'\cong M_n(k)\otimes(D\otimes L)\otimes L'\cong M_n(k)\otimes M_m(L)\otimes L'\cong M_{nm}(L').$$ This concludes the proof.
    \end{proof}
    
    We now work towards the construction of crossed product.
    
    \begin{definition}[factor set/cocycle]\index{factor set}\index{cocycle}\index{cocycle!group}
        Let $L$ be a Galois extension of $k$. $G=Gal(L/k)$. Then a map $\gamma: G\times G\to L^\times$, where $L^\times$ is the multiplicative group of $L$, is called \emph{factor set} or \emph{cocycle} of $G$ in $L^\times$ if it satisfies the equation
        \[ \gamma_{\sigma,\tau}\gamma_{\sigma\tau,\rho}=\sigma(\gamma_{\tau,\rho})\gamma_{\sigma,\tau\rho},  \ \sigma,\tau,\rho\in G\]
    The functions form a group called \emph{cocycle group} $Z^2(G,L^\times)$.\index{$Z^2(G,L^\times)$}
    \end{definition}
    
    \begin{definition}[crossed product]\index{crossed product}
        Let $L/k$ be Galois extension and $G=Gal(L/k)$. $\gamma$ is a factor set of $G$ in $L^\times$. Then the \emph{crossed product} $A\defeq (G,L,\gamma)$ consists of formal linear combinations $\sum_{\sigma\in G} a_\sigma u_\sigma$ for $a_\sigma\in L$, $u_\sigma$ are formal symbols indexed by $\sigma\in G$. 
        
        An algebra structure is defined on $A$ with coordinate-wise addition and the multiplication is determined by $$u_\sigma a=\sigma(a)u_\sigma \quad u_\sigma u_\tau=\gamma_{\sigma,\tau}u_{\sigma\tau}\quad \text{for all } a\in L,\ \sigma,\tau\in G.$$ Elements in $L$ multiplies in the usual way as in $L$.
    \end{definition}
    
    \begin{remark}
        The algebra is associative by the definition of the factor set, as one could verify directly.
    \end{remark}
    
    \begin{theorem}\label{4.1.7}
        In the above setting, $A=(G,L,\gamma)$ is central simple over $k$. $C_A(L)=L$ and $L$ is a splitting field of $A$. Moreover, $[A:k]=[L:k]^2$.
    \end{theorem}
    
    \begin{proof}
        If $\sum_\sigma a_\sigma u_\sigma\in Z(A)$, then for $a\in L$, $$\sum_\sigma (aa_\sigma) u_\sigma=a\sum_\sigma a_\sigma u_\sigma=(\sum_\sigma a_\sigma u_\sigma)a=\sum_\sigma a_\sigma (u_\sigma a)=\sum_\sigma a_\sigma \sigma(a)u_\sigma.$$ This means that whenever $a_\sigma\neq 0$, $\sigma(a)=a$. This forces $\sigma=1$. Thus, $C_A(L)=L$ and $Z(A)\subseteq L$. Moreover, if $a\in Z(A)$, $au_\sigma=u_\sigma a=\sigma(a)u$ for all $\sigma\in G$ Thus, $a\in INV(G)\cong k$, where $INV(G)=\{a\in L: \sigma(a)=a\text{ for all }\sigma\in G\}$(c.f. 5.4.4 \cite{drozd2012finite}). $Z(A)=k$. $A$ is central.
        
        Now let $I\ideal A$. Choose $x=\sum_{\sigma\in G} a_\sigma u_\sigma\in I$ with the least number of non-zero coefficient $a_\sigma$. By multiplying a suitable $\sigma\in G$, we can assume $a_1\neq 0$. Take any $a\in L$ then $ax-xa\in I$, but $$ax-xa=\sum_\sigma aa_\sigma u_\sigma-\sum_\sigma a_\sigma u_\sigma a=\sum_\sigma aa_\sigma u_\sigma-\sum_\sigma a_\sigma \sigma(a)u_\sigma=\sum_\sigma (aa_\sigma-a_\sigma\sigma(a)) u_\sigma.$$ so $aa_1-a_1\cdot 1(a)=0$. Then the number of non-zero coordinates in $ax-xa$ is smaller then $x$. Hence $ax=xa$, $x\in C_A(L)=L$ and $x$ is invertable. This forces $I=A$. $A$ is simple. $L$ is a splitting field by $(iv)$ of \ref{3.2} because $C_A(L)=L$. We also see $L$ is strictly maximal, so $[A:k]=[L:k]^2$.
    \end{proof}

    \begin{corollary}\label{4.1.8}
        Let $D$ be a finite dimensional central division algebra, and $L$ be its Galois splitting field. Then there exist $\gamma\in Z^2(G,L^\times)$ such that $(G,L^\times, \gamma)\sim D$.
    \end{corollary}
    
    \begin{proof}
        Let $D\otimes L\cong M_d(L)$ for $d=[D:k]^{1/2}$. $D\otimes L$ is a semisimple algebra and consider a simple module $U$. $U$ is a left $D$ right $L$ bimodule. Regard $U$ as a vector space over $D$. Then the right multiplication by $L$ is a $D$-linear endomorphism of $D$. Denote this matrix $T(\alpha)$. $T:L\to A$ where $A\defeq M_m(D)$, $m=[U:D]$. Hence $L$ is a subalgebra of $A$.
        
        Now $[U:k]=[U:D][D:k]=md^2$, but $U=L^d$ so $[U:k]=d[L:k]$. This means $[L:k]=md$ and $[A:k]=m^2d^2=[L:k]^2$. $L$ is a strictly maximal subfield in $A$, so $C_A(L)=L$.
        
        Take $\sigma\in Gal(L/k)$. By Skolem-Noether \ref{3.2.1}, there is an invertable element $u_\sigma\in A$ such that $\sigma(a)=u_\sigma a u_\sigma^{-1}$, i.e., $u_\sigma a=\sigma(a)u_\sigma$ for all $a\in L$. $u_\sigma$ is determined up to $C_A(L)=L$.
        
        Now take $\tau\in Gal(L/k)$. Then $\sigma\tau(a)=u_{\sigma\tau}au_{\sigma\tau}^{-1}$,$$\ \sigma(\tau(a))=u_\sigma u_\tau a u_\tau^{-1}u_\sigma^{-1}=(u_\sigma u_\tau)a(u_\sigma u_\tau)^{-1}.$$ Hence $u_{\sigma\tau}=\gamma_{\sigma,\tau}u_\sigma u_\tau$ for $\gamma_{\sigma,\tau}\in C_A(L)=L$.
        
        The associativity of multiplication in $A$ produce the precise condition for $\gamma_{\sigma, \tau}$ to satisfy the definition of factor set (just calculate $u_\sigma u_\tau u_\rho$ in two ways). Thus we have a homomorphism $\phi:(G,L,\gamma)\to A=M_m(D)$. By \ref{4.1.7}, $\phi$ is a monomorphism, and $[(G,L,\gamma):k]=[L:k]^2=[A:k]$. Now $\phi$ is an isomorphism.
    \end{proof}
    
    \begin{remark}
        This corollary, together with corollary \ref{4.1.4}, gives that any finite dimensional central division algebra is similar to some crossed products.
    \end{remark}
    
    \begin{definition}[coboundary]\index{coboundary}
        $\delta\in Z^2(G,L^\times)$ is called \emph{coboundary} if there is a function $\mu:G\to L^\times$ such that $$\delta_{\sigma,\tau}=\mu_\sigma\sigma(\mu_\tau)\mu_{\sigma\tau}^{-1},$$ for any $\sigma$, $\tau\in G$.
        
        The coboundaries form a subgroup $B^2(G,L^\times)$ of $Z^2(G,L^\times)$. We define $$H^2(G,L^\times)\defeq \frac{Z^2(G,L^\times)}{B^2(G,L^\times)}$$.\index{$B^2(G,L^\times)$}\index{$H^2(G,L^\times)$}
    \end{definition}
    
    \begin{theorem}\label{4.1.10}
    Let $\gamma, \eta\in Z^2(G,L^\times)$. The algebras $(G,L,\gamma)\cong (G,L,\eta)$ if and only if $\gamma=\eta\delta$ for some $\delta\in B^2(G,L^\times)$.
    \end{theorem}
    
    \begin{proof}
        Let $A=(G,L,\gamma)$, $B=(G,L,\eta)$. 
        
        \textit{(only if)}. Suppose $\phi:A\xrightarrow{\sim} B$ is the isomorphism. $\phi(L)$ is a subfield of $B$ isomorphic to $L$. By Skolem-Noether \ref{3.2.1}, we can without loss of generality assume $\phi(a)=a$, for all $a\in L$, by composing isomorphisms. Write $v_\sigma=\phi(u_\sigma)\in B$, where $u_\sigma$ is the element in $A$ indexed by $G$ as before. Let $u'_\sigma$ be the element in $B$ indexed by $G$ as before. $$v_\sigma a=\phi(u_\sigma a)=\phi(\sigma(a)u_\sigma)=\sigma(a)v_\sigma.$$ It follows that $v_\sigma=l_\sigma u'_\sigma$ for $l_\sigma\in C_B(L)=L$.
        
        Now $$\phi(u_\sigma u_\tau)=\phi(\gamma_{\sigma,\tau}u_{\sigma\tau})=\gamma_{\sigma,\tau}v_{\sigma\tau},$$ $$\phi(u_\sigma u_\tau)=\phi(u_\sigma)\phi(u_\tau)=v_\sigma v_\tau=l_\sigma u'_\sigma l_\tau u'_\tau=l_\sigma\sigma(l_\tau)l_{\sigma\tau}^{-1}\eta_{\sigma,\tau}v_{\sigma\tau}.$$
        Therefore $\gamma_{\sigma,\tau}=l_\sigma\sigma(l_\tau)l_{\sigma\tau}^{-1}\eta_{\sigma,\tau}\in \eta B^2(G,L^\times).$
        
        \textit{(if)} Let $\gamma=\eta\delta$ where $\delta_{\sigma,\tau}=l_\sigma\sigma(l_\tau)l_{\sigma\tau}^{-1}.$ Define $\phi':A\to B$, $\sum a_\sigma u_\sigma\mapsto \sum a_\sigma l_\sigma u'_\sigma$. This is an algebra homomorphism by direct verification. $A$ is simple and $[A:k]=[B:k]$. This shows that $\phi$ is an isomorphism.
        
    \end{proof}
    
    The following proof requires a version of the \emph{normal basis theorem} (c.f. 5.5.2 \cite{drozd2012finite}) in Galois theory and \emph{Pierce decomposition} (c.f. 1.7 \cite{drozd2012finite}).

    \begin{lemma}\label{4.1.11}
    $(G,L,\gamma)\otimes (G,L,\eta)\cong M_n((G,L,\gamma\eta))$, $n=[L:k]$.
    \end{lemma}
    
    \begin{proof}
        Let $A=(G,L,\gamma)$, $B=(G,L,\eta)$. Since $L$ is a subfield of $A$ and $B$, $L\otimes L$ is a subalgebra of $A\otimes B$. Let $_\sigma L$ be the $L\otimes L$-module with the action $(a\otimes b)\cdot x=\sigma(a)\cdot x \cdot b$. Then $L\otimes L$ decomposes as $n=[L:k]=|G|$ one-dimensional $L\otimes L$-modules, i.e., $L\otimes L\cong \oplus_{\sigma\in G}\ _\sigma L$ as $L\otimes L$-module (c.f. 5.4.1 \cite{drozd2012finite}). By Pierce decomposition (c.f. 1.7.2 \cite{drozd2012finite}), we know that there exist a unique non-zero idempotent $f\in L\otimes L$ such that $(a\otimes 1)f=f(1\otimes a)$, for all $a\in L$. Now we claim $(u_\sigma\otimes u_\sigma)f=f(u_\sigma\otimes u_\sigma)$ for any $\sigma\in G$. To see that, consider the idempotent $(u_\sigma\otimes u_\sigma)^{-1}f(u_\sigma\otimes u_\sigma)$ and check that $$(a\otimes 1)(u_\sigma\otimes u_\sigma)^{-1}f(u_\sigma\otimes u_\sigma)=(u_\sigma\otimes u_\sigma)^{-1} (\sigma(a)\otimes 1)f(u_\sigma\otimes u_\sigma)$$ $$=(u_\sigma\otimes u_\sigma)^{-1}f(1\otimes\sigma(a))(u_\sigma\otimes u_\sigma)=(u_\sigma\otimes u_\sigma)^{-1}f(u_\sigma\otimes u_\sigma)(1\otimes a).$$ Therefore by uniqueness, $(u_\sigma\otimes u_\sigma)^{-1} f(u_\sigma\otimes u_\sigma)=f.$
        
        Now let $T=f(A\otimes B)f$. We claim that $T\cong (G,L,\gamma\eta)$.
        
        Define $\phi:(G,L,\gamma\eta)\to T$, $$a\mapsto \bar{a}\defeq f(1\otimes a)=(a\otimes 1)f,$$ $$u_\sigma\mapsto \overline{u_\sigma}=f(u_\sigma\otimes u_\sigma)=(u_\sigma\otimes u_\sigma)f.$$ Then
        $$\bar{a}\bar{b}=\overline{ab},$$
        $$\overline{u_\sigma}\bar{a}=(u_\sigma\otimes u_\sigma)f\cdot f(1\otimes a)=(u_\sigma\otimes u_\sigma)(a\otimes 1)f=(\sigma(a)\otimes 1)(u_\sigma\otimes u_\sigma)f=\overline{\sigma(a)}\overline{u_\sigma}$$
        \begin{align*}
            \overline{u_\sigma}\overline{u_\tau}&=(u_\sigma\otimes u_\sigma)(u_\tau\otimes u_\tau)f=((u_\sigma u_\tau)\otimes (u_\sigma u_\tau))f=(\gamma_{\sigma,\tau}u_{\sigma\tau}\otimes \eta_{\sigma,\tau}u_{\sigma\tau})f\\
            &=(\gamma_{\sigma,\tau}\otimes 1)(1\otimes \eta_{\sigma,\tau})f(u_{\sigma\tau}\otimes u_{\sigma\tau})=(\gamma_{\sigma,\tau}\eta_{\sigma,\tau}\otimes 1)f(u_{\sigma\tau}\otimes u_{\sigma\tau})\\
            &=\overline{\gamma_{\sigma,\tau}}\overline{\eta_{\sigma,\tau}}\overline{u_{\sigma\tau}}
        \end{align*}
        
        Therefore, $\phi$ is a homomorphism, and hence a monomorphism. Now by Pierce decomposition, considering $(A\otimes B)f$ as a left $A\otimes B$-module, $f(A\otimes B)f$ is along the diagonal of the matrix of Pierce component (c.f. p.27 \cite{drozd2012finite}), so $T\cong End_{A\otimes B}((A\otimes B)f)$. Consider again the Pierce decomposition of $L\otimes L$. Its corresponding decomposition of identity in $L\otimes L$ is $1=\sum_{\sigma\in G} e_\sigma$ where $e_\sigma$ is the unique idempotent such that $ae_\sigma=\sigma(a)e_\sigma$. Moreover, for any $x=\sum a_i\otimes b_i \in L\otimes L$, $\tau x=\sum a_i \otimes \tau b_i$ but $\tau e_\sigma=e_{\tau \sigma}$ for $\tau \in G \subseteq LG$ (c.f. 5.5.2 \cite{drozd2012finite}), where $\tau$, considered as element of $LG$, acts on the second elements of $x$ via the action of $G$ on $L$. Write $e_\sigma=\sum a_i\otimes b_i\in L\otimes L$. We have $$(1\otimes u_\tau)e_\sigma = \sum a_i\otimes u_\tau b_i = \sum a_i\otimes \tau(b_i)u_\tau=\tau e_\sigma(1\otimes u_\tau)=e_{\tau\sigma}(1\otimes u_\tau).$$ Therefore $(1\otimes u_\tau)e_\sigma(1\otimes u_\tau)^{-1}=e_{\tau\sigma}$. This means all modules of $(A\otimes B)e_\sigma$ are isomorphic and in particular, isomorphic to $(A\otimes B)f$, $f=e_1$. Therefore, $$A\otimes B\cong \oplus_{\sigma\in G} (A\otimes B)e_\sigma\cong [(A\otimes B)f]^n,$$ $$A\otimes B\cong End_{A\otimes B}(A\otimes B) \cong M_n(T).$$ Hence, $$[T:k]=n^2=[(G,L, \gamma\eta):k],$$ so $\phi$ is an isomorphism.
    \end{proof}

    \begin{corollary}\label{4.1.12}
    
    $Br(L/k)\cong H^2(G,L^\times).$
    
    \end{corollary}
    
    \begin{proof}
        It follows from \ref{4.1.11} that the map $Z^2(G,L^\times)\to Br(L/k)$, $\gamma\mapsto [(G,L,\gamma)]$ is a well-defined homomorphism. It is kernel is $B^2(G,L^\times)$ by \ref{4.1.10} and its surjectivity dues to \ref{4.1.8}. It then implies $H^2(G,L^\times)\xrightarrow{\sim} Br(L/k)$, $\gamma B^2(G,L^\times)\mapsto [(G,L,\gamma)]$ is an isomorphism.
    \end{proof}
    
    \begin{theorem}\cite{artin1944rings}\label{4.1.13}
    Let $k\subseteq L\subseteq E$ be three fields with $E$, $L$ Galois over $k$. Let $G=Gal(E/k)$. $H=Gal(E/L)$. Then $(G,E,\gamma)\sim (G/H,L,\gamma')$, where $\gamma_{\sigma,\tau}=\gamma'_{\bar{\sigma},\bar{\tau}}$, where $\bar{\sigma}$, $\bar{\tau}$ are the images under the canonical quotient map $G\to G/H$.
    \end{theorem}
    
    \begin{proof}\cite{artin1944rings}
        Let $m=[E:L]$, $A'=(G/H,L,\gamma')$. Consider $A\defeq A'\otimes M_m(k)$. We want to show $A\cong (G,E,\gamma)$. 
        
        First we need to find a copy of $E$ in $A$. Consider $E$ as an $L$-vector space with basis $v=\{v_1,\cdots,v_m\}$. Consider $E$ to be a left regular $E$-module and right $L$-vector space. Then every $\alpha \in E$ induces a right $L$-linear endomorphism of $E$ and with respect to the basis $v$, has a unique matrix representation $M(\alpha)$, where $M:E\to M_m(L)\subseteq A$, \begin{equation}
            \alpha\cdot v=v\cdot M(\alpha)  \label{4.1.13.1}
        \end{equation}  
        $\alpha\in E$. $M$ is injective so $M(E)$ form a subring $E'$ of $A$ which is isomorphic to $E$.
        
        Similarly, for $\tau\in G$, $\tau(v)=\{\tau(v_1),\cdots, \tau(v_m)\}$, denote the corresponding matrix $P_\tau\in M_m(L)$ satisfying 
        \begin{equation}\label{4.1.13.2}
            \tau(v)=v\cdot P_\tau.  
        \end{equation}
        For any $\sigma\in G$, applying to equation (\ref{4.1.13.2}), we get $$vP_{\sigma\tau}=\sigma(\tau(v))=\sigma(vP_\tau)=\sigma(v)\sigma(P_\tau)=vP_\sigma\sigma(P_\tau),$$ hence,
        \begin{equation}\label{4.1.13.3}
           P_{\sigma\tau}=P_\sigma\sigma(P_\tau).
        \end{equation}
        
        Applying $\sigma\in G$ to equation (\ref{4.1.13.1}) we get
        $$\sigma(\alpha v)=\sigma(\alpha)\sigma(v)=\sigma(\alpha)\cdot v P_\sigma=vM(\sigma(\alpha))P_\sigma$$
        $$\sigma(vM(\alpha))=\sigma(v)\sigma(M(\alpha))=vP_\sigma\sigma(M(\alpha))$$
        hence,
        \begin{equation}\label{4.1.13.4}
            P_\sigma\sigma(M(\alpha))=M(\sigma(\alpha))P_\sigma,
        \end{equation}
        where $\sigma(M(\alpha))$ is mapped entrywisely. 
        
        Since $M(\alpha)$, $P_\sigma\in M_m(L)$, $\sigma(M(\alpha))=\bar{\sigma}(M(\alpha))$, $\sigma(P_\sigma)=\bar{\sigma}(P_\sigma)$, where $\bar{\sigma}$ is the image of $\sigma\in G$ in $G/H$.
        
        Now we want to construct a homomorphism from $(G,E,\gamma)$ to $A$. Let $u_{\bar{\sigma}}$ be the element in $A'$ satisfying $$ u_{\bar{\sigma}} a = \bar{\sigma}(a) u_{\bar{\sigma}},$$
        $$u_{\bar{\sigma}}u_{\bar{\tau}}=\gamma'_{\bar{\sigma},\bar{\tau}}u_{\bar{\sigma}\bar{\tau}}$$ for $a\in L$. Set $u'_\sigma=P_\sigma u_{\bar{\sigma}}$ in $A$. Then
        $$ u'_\sigma M(\alpha)=P_\sigma u_{\bar{\sigma}} M(\alpha)=P_\sigma \bar{\sigma}(M(\alpha))u_{\bar{\sigma}}\stackrel{(\ref{4.1.13.4})}{=}M(\sigma(\alpha))P_\sigma u_{\bar{\sigma}}=M(\sigma(\alpha)) u'_{\sigma}.$$
        $$ u'_\sigma u'_\tau =P_\sigma u_{\bar{\sigma}} P_\tau u_{\bar{\tau}}=P_\sigma \bar{\sigma}(P_\tau) u_{\bar{\sigma}} u_{\bar{\tau}} = P_\sigma\bar{\sigma}(P_\tau) \gamma'_{\bar{\sigma},\bar{\tau}} u_{\bar{\sigma}\bar{\tau}} \stackrel{(\ref{4.1.13.3})}{=} \gamma'_{\bar{\sigma},\bar{\tau}} P_{\sigma\tau} u_{\overline{\sigma\tau}}=\gamma'_{\bar{\sigma},\bar{\tau}} u'_{\sigma\tau},$$ 
        where $\bar{\sigma}$, $\bar{\tau}$ is the image of $\sigma$, $\tau\in G$ in $G/H$.
        Therefore there exist a homomorphism from $(G,E,\gamma)$ to $A$.
        
        $(G,E,\gamma)$ is simple, and $$[(G,E,\gamma):k]=[E:k]^2=m^2[L:k]^2=[M_m(k):k][A':k]=[A:k]$$ gives bijectivity. Hence $(G,E,\gamma)\cong A\sim A'=(G/H, L, \gamma')$.
        \end{proof}
        
    \section{The Cyclic Algebra}
            
            The simplest type of crossed product is when the Galois group $G$ is cyclic. This section is mainly based on \cite{jacobson2009}, unless explicitly cited.
            
        \begin{definition}[cyclic extension]\index{cyclic extension}
            A field extension $L/k$ is \emph{cyclic} if $L/k$ is Galois and $Gal(L/k)$ is a cyclic group.
        \end{definition}
            
        \begin{definition}[cyclic algebra]\label{4.2.2}\index{cyclic algebra}
            Let $L/k$ be a cyclic extension with $Gal(L/k)=\left<\sigma\right>$. Then the crossed product $A=(\left<\sigma\right>,L,\gamma)$ is called a \emph{cyclic algebra} where 
            \[ \gamma_{\sigma^i,\sigma^j}=
               \left\{
                    \begin{array}{lll}
                          1& if\ 0\leq i+j<n &\\
                        r& if\ i+j\geq n & for \ 0\leq i,j<n
                    \end{array} 
                \right. 
            \]
            For notational convenience, we write $A=(\sigma,L,r)$ in case of cyclic algebra.\index{$(\sigma,L,r)$}
        \end{definition}
        
        Now we want to construct a cyclic algebra from a crossed product $A=(G,L,\gamma)$ where $G=\left<\sigma\right>$ is cyclic \cite{jacobson2009}. Let $n=|G|$. By the multiplication law, $u_\sigma a=\sigma(a)u_\sigma$ so $\sigma(a)=u_\sigma a u_\sigma^{-1}$. Write $u=u_\sigma$. Inductively for $0\leq i<n$, $\sigma^i(a)=u^i a u^{-i}$ so $u^i a=\sigma^i(a)u^i$. But $u_{\sigma^i}a=\sigma^i(a)u_{\sigma^i}$. It follows that $(u_{\sigma^i}^{-1}u^i)a=a(u_{\sigma^i}^{-1}u^i)$. $(u_{\sigma^i}^{-1}u^i)\in C_A(L)=L$ so $u^i=l_i u_{\sigma^i}$ for $l_i\in L^\times$. 
        
        Now $\{u^i\}_{0\leq i<n}$ is also an $L$-basis of $A$. We can replace $u_{\sigma^i}$ by $u^i$ for $0\leq i<n$. Denote the corresponding factor set $\gamma'$. For $i+j<n$, $u^{i+j}$ is the basis element, so $u^i u^j=u^{i+j}$, and $\gamma'_{i,j}=1$. For $i+j>n$, we want to find out $\gamma_{i,j}$, but $u^{i+j}$ is not necessarily a basis element so we need some more calculation. Note that $u^i u^j=u^{i+j}=u^n u^{i+j-n}$. We claim $u^n\in L^\times$. $$u^n a=u_\sigma\cdots u_\sigma a=\sigma(a)^n u_\sigma\cdots u_\sigma =\sigma^n(a)u^n.$$ Since $\sigma^n=1$, $u^n a=a u^n$ for all $a\in L$ so $u^n\in C_A(L)=L$. Denoting $r\defeq u^n$, we have $\gamma'_{i,j}=r$ for $i+j\geq n$. Now we have
        \[ \gamma'_{i,j}=
               \left\{
                    \begin{array}{lll}
                          1& if\ 0\leq i+j<n &\\
                        r& if\ i+j\geq n & for \ 0\leq i,j<n
                    \end{array} 
                \right. 
            \]
        and the multiplication in $A=\oplus_{i=0}^{n-1} Lu^i$ with respect to this new basis is
        $$ua=\sigma(a)u,\ \ u^n=r.$$
        Since $u^n$ commutes with $u$ and every element of $L$, $u^n\in Z(A)=k$, $r\in k$.
        
        The cyclic algebra inherits properties from the crossed product. First we need a definition.
        
        \begin{definition}[field norm]\index{field norm}\index{$N(L^\times)$}\index{$N_{L/k}(a)$}
            Let $L/k$ be a finite field extension. Then the \emph{field norm} $N_{L/k}$ of $a\in L$ is defined as
                         $$N_{L/k}(a)=det([\rho(a)])$$
            where $[\rho_a]$ is the matrix of regular representation $\rho_a:x\mapsto ax$ in $L$.
            
            We denote $N(L^\times)=\{N_{L/k}(a): \ a\in L^\times\}$ as a subgroup of $k^\times$.
        \end{definition}
        
        \begin{lemma}\label{4.2.3}
            If $L/k$ is a field extension. $N_{L/k}$ is a field norm of $L/k$. Then
            \begin{enumerate}[(i)]
                
                \item Define the characteristic polynomial of $[\rho_a]$ to be $\chi([\rho_a])(x)\defeq det(xI-[\rho_a])$. Write $\chi([\rho_a])(x)=x^n-a_1 x^{n-1}+\cdots+(-1)^n a_n$. Then $$N_{L/k}(a)=(-1)^n \chi([\rho_a])(0)=a_n\in k;$$
                \item $\chi([\rho_a])(x)=(m_a(x))^r$, where $r=n/m$, $m=[k(a):k]$, $m_a(x)$ is the minimal polynomial of $a$ over $k$ of degree $m$;
                \item Let $m_a(x)=\prod_{i=1}^m (x-u_i)$ be a factorization of $m_a(x)$ in the splitting field. Then an equivalent definition to the norm is $$N_{L/k}(a)=(\prod_{i=1}^m u_i)^{r};$$
                \item $N_{L/k}: L^\times\to k^\times$ is a group homomorphism from the multiplicative group of L to the multiplicative group of k.
                $$N_{L/k}(ab)=N_{L/k}(a)N_{L/k}(b)\quad N_{L/k}(sa)=s^n N_{L/k}(a)\quad  a,b\in L,\ s\in k$$ where $n=[L:k]$.
                In particular, if $L/k$ is Galois, $$N_{L/k}(a)=\prod_{\sigma\in Gal(L/k)} \sigma(a);$$
                \item For any $\sigma\in Aut_k L$, $$N_{L/k}(\sigma(a))=N_{L/k}(a)=\sigma N_{L/k}(a).$$ 
            \end{enumerate}
            \begin{proof}
            One can find a proof in many classical texts, say in p.p. 9-11 \cite{ash2010ANT}.
            \end{proof}

        \end{lemma}

        \begin{theorem}\label{4.2.4}
            $k^\times/N(L^\times)\cong Br(L/k)$ via the map $rN(L^\times)\mapsto [(\sigma, L, r)]$.
        \end{theorem}
        
        \begin{proof}
        This theorem follows from the isomorphism $H^2(G,L^\times)\cong Br(L/k)$ in \ref{4.1.12} if we can show $k^\times/N(L^\times)\cong H^2(G,L^\times)$ in the case $G=\left<\sigma \right>$, $\gamma_r$ is the factor set in $(\sigma, L, r)$ in \ref{4.2.2}. Let $\phi:k^\times\to H^2(G,L^\times)$, $r\mapsto \gamma_r B^2(G,L^\times)$, where 
        \[ \gamma_{r_{\sigma^i,\sigma^j}}=
               \left\{
                    \begin{array}{lll}
                          1& if\ 0\leq i+j<n &\\
                        r& if\ i+j\geq n & for \ 0\leq i,j<n
                    \end{array} 
                \right. 
            \]
        Take $w=\prod_{i=1}^n \sigma^i(a)\in N(L^\times)$, we write $\gamma=\gamma_w$. If $\gamma_{\sigma^i,\sigma^j}=\mu_{\sigma^i}\sigma(\mu_{\sigma^j})\mu_{\sigma^{i+j}}^{-1}$, we have $\mu_{\sigma^{i+j}}=\gamma_{\sigma^i,\sigma^j}^{-1}\mu_{\sigma^i}\sigma(\mu_{\sigma^j})$. $\mu_1=\mu_1\sigma(\mu_1)$, which means $\mu_1=1$. For $i<n$, $\mu_{\sigma^i}=\mu_{\sigma^{1+(i-1)}}=\mu_\sigma \sigma(\mu_{\sigma^{i-1}})$ and inductively, $\mu_{\sigma^i}=\prod_{j=0}^{i-1} \sigma^j(\mu_\sigma)$.
        For $i=n$, $\mu_{\sigma^n}=\gamma_{\sigma,\sigma^{n-1}}^{-1}\mu_{\sigma}\sigma(\mu_{\sigma^{i-1}})=w^{-1} w=1$. Note that this is constant because $\mu_{\sigma^n}=\mu_1=1$ and this happens only when $\gamma_{\sigma,\sigma^{n-1}}=w\in N(L^\times)$. By properties of factor set we have $\gamma_w\in B^2(G,L^\times)$. Hence for all $w\in N(L^\times)$, $\gamma_w\in B^2(G,L^\times)$. We also see if $r\notin N(L^\times)$, such $\mu_{\sigma^n}\neq \mu_1$ and such $\mu$ is not well-defined, so $\gamma_r\notin B^2(G,L^\times)$. Therefore $ker(\phi)=N(L^\times)$. $\phi$ is obviously surjective, we have $k^\times/N(L^\times)\cong H^2(G,L^\times)$.
        \end{proof}
        
        \begin{corollary}\label{4.2.6}
            Let $E/k$ be a cyclic extension and $L$ be a subfield of $E$. Let $Gal(E/k)=\left<\sigma\right>$, then the cyclic algebras $(\bar{\sigma}, L, r)\sim (\sigma, E, r^m)$ where $m=[E:L]$, where $\bar{\sigma}$ restriction of $\sigma$ to $L$ that generates $Gal(L/k)$.
        \end{corollary}
            
        \begin{proof}
        This is a special case of theorem \ref{4.1.13}.
        Consider $(\bar{\sigma}, L ,r')$ such that $$(\sigma, E, r^m)\cong (\bar{\sigma}, L ,r')\otimes M_m(k)$$ by the isomorphism given in \ref{4.1.13}. We claim that $r'=r$.
        
        Now $G=Gal(E/k)$, $H=Gal(E/L)$, $n=|G|$, $m=|H|$, and $d=\frac{n}{m}=|Gal(L/k)|$.
        Denote $u_\sigma$, $\bar{u}_{\bar{\sigma}}$ to be the elements in $(\sigma, E, r^m)$ and $(\bar{\sigma}, L ,r')$ indexed by elements in $G$ and $H$, and $(u_\sigma)^n=r^m$, $(\bar{u}_{\bar{\sigma}})^d=r'$. Now let
        \begin{align*}
            \phi: (\sigma, E, r^m)  &\xrightarrow{\sim}  (\bar{\sigma}, L ,r')\otimes M_m(k)\\
                  u_\sigma     &\mapsto  \bar{u}_{\bar{\sigma}}P_\sigma.
        \end{align*}
        We have
        \begin{align*}
            \phi(r^m)&=\phi(u_\sigma^n)=(\bar{u}_{\bar{\sigma}}P_\sigma)^n=\prod_{i=1}^n \bar{\sigma}^i(P_\sigma)((\bar{u}_{\bar{\sigma}})^d)^m\\
            &=\prod_{i=1}^n \sigma^i(P_\sigma)r'^m\stackrel{\eqref{4.1.13.3}}{=}P_{\sigma^n}r'^m=(1\cdot P_1)r'^m\\
            &=\phi(r'^m)
        \end{align*}
        because $r, r'\in k$, and $k$ is fixed under $\phi$. This shows $r=r'$ and we are done with the proof.
        
        \end{proof}

        \chapter{Valuation and Local Fields}
        This chapter introduce the theory of valuation. The valuations are useful when studying local fields and their finite dimensional extensions. It is mainly based on \cite{jacobson2009}, unless explicitly cited. 
        \section{Valuation and Discrete valuation}

        \begin{definition}[ordered abelian group]\index{ordered abelian group}
        An \emph{ordered abelian group} is a pair $(G,H)$ where $G$ is an abelian group and $H$ is a subset of $G$ defining the order. It satisfies the following properties:
        \begin{enumerate}[(i)]
            \item $G$ is a disjoint union of $H\cup\{1\}\cup H^{-1}$, where $H^{-1}=\{h^{-1}: \ h\in H\}$;
            \item $H$ is multiplicatively closed in $G$.
        \end{enumerate}
        \end{definition}
        
        \begin{remark}
            $H$ is used to define an order in $G$ in the following way. If $g_1>g_2$, then $g_1^{-1}g_2\in H$. It is transitive, total (one and only one possibility: $g_1>g_2$, $g_1=g_2$, $g_1<g_2$), and multiplication compatible ($g_1>g_2$, $g_3>g_4$ $\implies$ $g_1 g_3>g_2 g_4$, $g_2^{-1}>g_1^{-1}$).
            Moreover, if we have an abelian group and a total order relation $<$, we can define $H=\{h: \ h<1\}$ and get an ordered abelian group. It can be easily seen that not every abelian group can be ordered, and if $g^n=1$ for an ordered $G$, then $g=1$. 
        \end{remark}
            
        Now for $(G,H)$, we adjoin $\{0\}$ to it and define $0\cdot 0=0$, $g>0$, $0\cdot g=0=g\cdot 0$, for all $g\in G$.
        
        \begin{definition}[valuation]\index{valuation}
            A \emph{valuation} of a field $k$ is a map $\phi:k\to G\cup \{0\}$, where $(G,H)$ is an ordered abelian group, satisfying
            \begin{enumerate}[(i)]
                \item $\phi(a)\geq 0$ for all $a$, and $\phi(a)=0$ if and only if $a=0$;
                \item $\phi(ab)=\phi(a)\phi(b)$;
                \item $\phi(a+b)\leq max\{\phi(a), \phi(b)\}$, $a,b\in k$.\quad (strict triangle inequality) 
            \end{enumerate}
        \end{definition}
            
        \begin{remark}
            When $(G,H)=R_{>0}^\times$, $\phi=|\cdot|$ is a \emph{non-Archimedean} absolute value.\index{absolute value!non-Archimedean}. If $(iii)$ is replaced by the usual triangle inequality, $|\cdot|$ is an \emph{Archimedean} \index{absolute value!Archimedean} absolute value.  We call $\phi(k^\times)$ the \emph{value group}, a subgroup of $G$. We can without loss of generality replace $G$ by $\phi(k^\times)$.\index{absolute value}
        \end{remark}
        
        \begin{definition}[discrete valuation]\index{valuation!discrete}
            A valuation is \emph{discrete} if the valuation group is cyclic.
        \end{definition}
        
        Here are some commonly used facts for valuation. Since $G$ has no element of finite order other than $1$, we have $\phi(a^n)=\phi(b^n)$ implies $a=b$. In particular, if $u$ is a unity of $1$, then $\phi(u)=1$, $\phi(-1)=\phi(1)=1$.
            
        \begin{definition}[equivalent valuations]\index{valuation!equivalent}
        Two valuations $\phi_1$, $\phi_2$ on $k$ are \emph{equivalent} if there exist an order-preserving map (order isomorphism) $\eta$ of value groups $\phi_1(k^\times)$, $\phi_2(k^\times)$ such that $\phi_2=\eta \phi_1$. This is obviously an equivalence relation of valuations.
        \end{definition}
        
        \section{Valuation Ring}
            There is an equivalent way of saying valuations of a field: The valuation ring. It is a subring of the field $k$.
            
        \begin{definition}[valuation ring]\index{valuation!ring}
            Let $k$ be a field and $R$ be a subring in $k$. $R$ is called a \emph{valuation ring} if for every $a\in k$, either $a\in R$ or $a^{-1}\in R$.
        \end{definition}
        
        \begin{lemma} \label{5.2.2}
        A valuation ring $R$ of a field $k$ is in one-to-one correspondence with the equivalent classes of valuations $[\phi]$ of $k$.
        \end{lemma}
        
        \begin{proof}
            Given a valuation $\phi$, define $R\defeq \{a\in k: \ \phi(a)\leq 1\}$ This is a subring because $1\in R$, $\phi(a-b)\leq max\{\phi(a),\phi(b)\}\leq 1$, $\phi(ab)=\phi(a)\phi(b)\leq 1$ for $a,b\in R$. If $a\notin R$, then $\phi(a)>1$, $\phi(a^{-1})=\phi(a)^{-1}<1$ so $a^{-1}\in R$. $R$ is a valuation ring. It is clear that $[\phi]$ corresponds to the same $R$.
            
            Now given a valuation ring $R$ in $k$. Let $U$ be the units in $R$. $P$ be the set of non-units of $R$. Let $R^\times=R\cap k^\times$, $P^\times=P\cap k^\times$. $U$ is a subgroup of $k^\times$. Form a factor group $G_0\defeq k^\times/U$. $H_0\defeq\{bU: \ b\in P^\times\}\subseteq G_0$.
            
            We claim $(G_0,H_0)$ is an ordered abelian group. If $a\in k^\times$, then $a\in R$ or $a^{-1}\in R$. There are three cases. If $a$, $a^{-1}\in R$, then $a\in U=1$. If $a\in R$, $a^{-1}\notin R$, then $a\in P^\times$, $aU\in H_0$. If $a\notin R$, then $a^{-1}\in P^\times$, $aU\in H_0^{-1}$. The first condition of ordered abelian group is proved.
            
            If $b_1$, $b_2\in P^\times$, then $b_1 b_2\in P^\times$, for otherwise if $cb_1 b_2=1$, $c\in R$, then $cb_1$ is an inverse to $b_1$. Hence $H_0$ is multiplicatively closed.
            
            Now define $\phi_0:k\to G_0\cup \{0\}$ by $\phi_0(0)=0$, $\phi_0(a)=aU$, for $a\neq 0$. To see $\phi_0$ is a valuation, first two conditions are straightforward. For $a,b \in k$, if either $a=0$ or $b=0$, $(iii)$ is satisfied. If $a\neq 0$, $b\neq 0$, then either $a^{-1}b\in R^\times$, or $b^{-1}a\in R^\times$. Assume $a^{-1}b \in R^\times$. $\phi_0(a^{-1}b)\leq 1$ since $R^\times=U\cup P^\times$, $\phi_0(a)\geq\phi_0(b)$. Since $a^{-1}b+1\in R$, $\phi_0(a^{-1}b+1)\leq 1$. Now $\phi_0(a+b)=\phi_0(a)\phi_0(a^{-1}b+1)\leq \phi_0(a)=max\{\phi_0(a),\phi_0(b)\}$. Thus $\phi_0$ is a valuation, and we call it the \emph{canonical valuation} of $R$.
            
            Now take the arbitrary valuation $\phi':k\to G'\cup \{0\}$, $(G',H')$ is the ordered abelian group. Let $R$ be the corresponding valuation ring and $\phi_0$ be the canonical valuation of $R$. We want to show $[\phi']=[\phi_0]$. We have $\phi_0(a)=aU$, $a\neq 0$ for $U$ the units in $R$. On the other hand, the kernel of the group homomorphism $\phi'|_{k^\times}$ is $U$. Hence $\phi'(k^\times)\cong k^\times/U\cong \phi_o(k^\times)=G_0$. Let $\eta$ be this isomorphism of $\phi'(k^\times)$ and $\phi_0(k^\times)$. To show this is order preserving take $\phi_0(b)=bU\in H_0$, then $b\in P^\times\subseteq R$. $\phi'(b)<1$. Hence $\eta$ is order preserving. This concludes the proof.
          \end{proof}

        \begin{lemma}\label{5.2.3}
            $R$ is a valuation ring of $k$ with respect to the valuation $\phi$. Then $R$ is a local ring with the unique maximal ideal $P\defeq \{a\in R: \ \phi(a)<1\}$ which consists of non-units of $R$.
        
        \end{lemma}
            
        \begin{proof}
        Take $x$ to be a non-unit in $R$. $x^{-1}\notin R$. $\phi(x^{-1})>1$, $\phi(x)=\phi(x^{-1})^{-1}<1$ so $x\in P$. Therefore $P$ is the set of non-units in $R$.
        
        Since $\phi(a+b)\leq max\{\phi(a),\phi(b)\}<1$, $\phi(ra)=\phi(r)\phi(a)\leq 1\cdot \phi(a)<1$ for $a,b\in P$, $r\in R$. Therefore $P$ is an ideal, and hence maximal ideal as all non-units are in $P$.
        
        Since every maximal ideal consists of non-units, they are equal to $P$. $P$ is the unique maximal ideal. $R$ is local.
        \end{proof}
        
        \begin{definition}[residue field]\index{residue field}
        In the above setting, $R/P$ is called the \emph{residue field} of the valuation $\phi$.
        \end{definition}
        
        \begin{theorem}\label{5.2.5}
        The following statements are equivalent.
        \begin{enumerate}[(i)]
            \item $R$ is a valuation ring of a discrete, non-trivial $|\cdot |$ in its field of fractions $k$;
            \item $R$ is a local P.I.D.
        \end{enumerate}
        \end{theorem}
            
        \begin{proof}
        ($(i)\implies (ii)$) Suppose $|\cdot|$ is discrete, non-trivial on a field $k$ and $R$ is the corresponding valuation ring. Let $c$ be a generator of $|k^\times|$ such that $c<1$. Let $\pi\in k^\times$ such that $|\pi|=c<1$ so $\pi\in P$, $P$ is the maximal ideal of $R$. 
        
        Now for every $a\in k^\times$, $|a|=c^s=|\pi|^s$ for some $s\in \Z$. This means $|a\pi^{-s}|=1$. Let $u\defeq a\pi^{-s}$ be a unit (\ref{5.2.3}), and $a=\pi^s u$. $a\in R$ if and only if $s\geq 0$. $a\in P$ if and only if $s>0$. Therefore $P=\pi R$ is principal. Now for any $I\ideal R$, choose $a\in I$ so that $|a|=|\pi|^t$ for $t$ as small as possible. $a=\pi^t u'$ for some unit $u'$. $\pi^t=au'^{-1}\in I$. $P^t\subseteq I$. Take $b\in I$, $b=\pi^l u''$ for $l\geq t$ by minimality of $t$, so $b=\pi^t\pi^{l-t}u''\in \pi^t R=P$. $I\subseteq P^t$. Hence $R$ is a P.I.D., and together with $\ref{5.2.3}$ $R$ is a local P.I.D.
        
        ($(ii)\implies (i)$) Suppose $R$ is a local P.I.D. Let $P$ be the maximal ideal. $P=\pi R$. The only primes in $R$ are $\pi$ and its associates (differ from $\pi$ by a unit). Hence, every non-zero element in $R$ is of the form $u\pi^k$, $k\geq 0$, $u$ a unit. Then every non-zero element of $k$ can be written as $v\pi^t$, $t\in \Z$, $v$ is a unit. Hence $R$ is a valuation ring of $k$ and the valuation is discrete.
        \end{proof}
            
        \section{Extension of Valuation}\label{5.3}
        
        The aim of this section is to show the existence and uniqueness of extension of $|\cdot|$ to an absolute value of the extension field.
        \begin{theorem}[uniqueness]\label{5.3.1}
            Let $|\cdot|$ be a non-trivial absolute value on $k$ with $k$ complete relative to it. $L/k$ is a finite field extension. If $|\cdot|$ can be extended to $|\cdot|$ in $L$, then it is necessarily unique. Moreover, $$|a|=|N_{L/k}(a)|^{1/[L:k]}\quad for\ a\in L,$$ and $L$ is complete relative to it.
        \end{theorem}
        
        \begin{proof}
            Take a $k$-basis of $L$: $\{u_1,\cdots,\ u_r\}$ with $r=[L:k]$. For a sequence $\{a_n\}$ in $L$, we can form $r$ sequences in $k$ by $a_n=\sum_{i=1}^r \alpha_{ni}u_r$, for $\alpha_{ni}\in k$.
            
            First we want to show $\{a_n\}$ is Cauchy if and only if $\{\alpha_{ni}\}$ is Cauchy for every $1\leq i\leq r$. Note for later that this only requires $u_1,\cdots u_r$ to be linearly independent.
            
            If $\{\alpha_{ni}\}$ is Cauchy for all $i$, then by properties of Cauchy sequence, $\{a_n\}$ is Cauchy.
            
            For the other direction we use induction on $r$.  Assume $\{a_n\}$ is Cauchy. When $r=1$ it holds trivially. In general if $\{\alpha_{nr}\}$ is Cauchy, $b_n\defeq a_n-a_{nr}u_r=\sum_{i=1}^{r-1} \alpha_{ni}u_r$ is Cauchy by induction and the result follows. Hence it remains to show $\{\alpha_{nr}\}$ is Cauchy. If not, there is a real $\epsilon>0$ such that for every $N\in \N$, there is a pair $(p,q)$, with $p,q\geq N$ such that $|\alpha_{pr}-\alpha_{qr}|>\epsilon$. We can form a sequence $(p_k,q_k)$ such that $|\alpha_{p_k r}-\alpha_{q_k r}|>\epsilon$ for all $k$, where $p_1<p_2<\cdots$, and $q_1<q_2<\cdots$. Therefore $(\alpha_{p_k r}-\alpha_{q_k r})^{-1}$ exist. Let 
            \begin{equation}\label{5.3.1.1}
                c_k\defeq \frac{\alpha_{p_k}-\alpha_{q_k}}{\alpha_{p_k r}-\alpha_{q_k r}}=\sum_{i=1}^{r-1} \frac{\alpha_{p_k i}-\alpha_{q_k i}}{\alpha_{p_k r}-\alpha_{q_k r}}u_i+u_r=\sum_{i=1}^{r-1} \beta_{ki}u_i+u_r=d_k+u_r
            \end{equation}
            where $$\beta_{ki}=\frac{\alpha_{p_k i}-\alpha_{q_k i}}{\alpha_{p_k r}-\alpha_{q_k r}},\quad d_k=\sum_{i=1}^{r-1} \beta_{ki}u_i.$$ Since $(\alpha_{p_k}-\alpha_{q_k})\to 0$, $(\alpha_{p_k r}-\alpha_{q_k r})^{-1}<\epsilon^{-1}$. We have $c_k\to 0$. Therefore $d_k\to -u_r$. $\{d_k\}$ is Cauchy. By induction hypothesis $\{\beta_{ki}\}$ is Cauchy for every $1\leq i<r$. Since $k$ is complete, let $\beta_{ki}\to \beta_i\in k$. Passing to limit in equation \eqref{5.3.1.1} we have $\sum_{i=1}^{r-1} \beta_{i}u_i+u_r=0$ but it contradicts with the fact that $\{u_i\}$ is linearly independent. Hence $\{\alpha_{nr}\}$ is Cauchy and we proved the first claim.
            
            Now it is clear that $L$ is complete with respect to $|\cdot|$.
            
            It remains to prove $|a|=|N_{L/k}(a)|^{1/r}$ if $|\cdot|$ exist. Suppose not, by switching $a$ to be $a^{-1}$ we may assume $|a|^r<|N_{L/k}(a)|$. Consider $b=a^r N_{L/k}(a)^{-1}$. $|b|<1$ but $N_{L/k}(b)=\frac{N_{L/k}(a)}{N_{L/k}(a)}=1$. Since $|b|<1$, $b^n\to 0$. If we write $b^n$ in terms of basis $b_n=\sum_{i=1}^r l_i u_i$ then the arguments for Cauchy sequence implies $l_i\to 0$ for all $i$. Now because the norm map $N_{L/k}(\cdot)$ is continuous, $1=N_{L/k}(b^n)\to 0$ which gives a contradiction. The $|\cdot|$ is therefore unique, if exist.
         \end{proof}
            
        \begin{lemma}\label{5.3.2}
            Let $L/k$ be a finite field extension, $[L:k]=n$. Let $\phi$ be a valuation of $L$. Then the value group of $L$ is order isomorphic to a subgroup of the value group of $k$.
        \end{lemma}
        
        \begin{proof}
            Because the extension is finite, every $a\in L$ satisfies a relation of the form $\alpha_1a^{n_1}+\cdots+\alpha_ka^{n_k}=1$ where $\alpha_i$ are non-zero elements in $k$ and $n\geq n_1>n_2>\cdots>n_k$. If for some $i$, $\phi(\alpha_i a^{n_i})>\phi(\alpha_j a^{n_j})$ for all $j$ with $j\neq i$, then $\phi(\sum_{j=1}^k \alpha_j a^{n_j})=\phi(\alpha_i a^{n_i})$, but $\phi(\sum_{j=1}^k \alpha_j a^{n_j})=0$, $\phi(\alpha_i a^{n_i})\neq 0$, we get a contradiction. Therefore there exist some pair $(i,j)$ with $i>j$ such that $\phi(\alpha_i a^{n_i})=\phi(\alpha_j a^{n_j})$. Then $\phi(a^{n_i-n_j})=\phi(\alpha_j\alpha_i^{-1})\in \phi(k^\times)$. Therefore all multiples of $\phi(a)^{n_i-n_j}$ is in $\phi(k^\times)$. In particular, for every $a \in L$, $\phi(a)^{n!}$ is always in $\phi(k^\times)$ since $n_i-n_j\leq n$. As $\phi(k^\times)$ has no finite order elements except $1$, the map $\phi(L^\times)\to \phi(k^\times)$, $g\mapsto g^{n!}$ is an order monomorphism of value groups. Hence $\phi(L^\times)$ is order isomorphic to a subgroup of $\phi(k^\times)$.
         \end{proof}
            
        \begin{theorem}[Existence, non-Archimedean case (Theorem 9.12 \cite{jacobson2009})]
            Let $k$ be a field with non-Archimedean absolute value $|\cdot|$. $L/k$ be a finite field extension. Then there exist an extension of absolute value $|\cdot|$ on $k$ to an absolute value $|\cdot|$ on $L$.
        \end{theorem}
            
        \begin{proof}
        The construction is similar to that of \ref{5.2.2} containing unexcited details, which we will omit here. The reader may refer to Theorem 9.12 of \cite{jacobson2009} p.p. 580-583 for a complete proof.
        \end{proof}
    
    \section{Unramified Extension and Completely Ramified Extension}
    
    This section introduces a important construction: The unramified extension, which comes in pair with the completely ramified extension, and both of them play crucial roles in determining the Brauer group of local fields.
    
    Throughout this session, unless specified, we will use the setting below.
    
    Let $L/k$ be a field extension. Let $\phi$ be a valuation of $L$, and $S$ be its valuation ring in $L$. Denote $Q$ to be the maximal ideal of $S$. $R$ is the valuation ring for $k$ and $P$ is its maximal ideal. Then $R=S\cap k$, $P=Q\cap k$. Let $\bar{S}\defeq S/Q$, $\bar{R}\defeq R/P$. Then we have an monomorphism $\bar{R}\hookrightarrow \bar{S}$, $a+P\mapsto a+Q$. Therefore, we can consider $\bar{R}$ embeded into $\bar{S}$. Also, $\phi(k^\times)$ is a subgroup of the value group $\phi(L^\times)$.
    
    \begin{align*}
        &L\geq S\idealinv Q& \quad &\bar{S}=S/Q&\\
        &|\\
        &k\geq R\idealinv P& \quad &\bar{R}=R/P&\\
    \end{align*}
            
    \begin{definition}[ramification index and residue degree]\index{ramification index}\index{residue degree}
        In the above setting, the \emph{ramification index} of the field extension $L/k$ with respect to the valuation $\phi$ is $$e\defeq [\phi(L^\times):\phi(k^\times)].$$
        The \emph{residue degree} of $L/k$ with respect to $\phi$ is $$f\defeq [\bar{S}:\bar{R}].$$
    \end{definition}    
            
    \begin{definition}[unramified and completely ramified field extension]\index{unramified field extension}\index{completely ramified field extension}
        In the above setting, $L$ is \emph{unramified} over $k$ if the ramification index $e=1$, i.e., $\phi(L^\times)=\phi(k^\times)$, the value group is unchanged.
        
        $L$ is \emph{completely ramified} over $k$ if the residue degree $f=1$, i.e., $\bar{R}\cong \bar{S}$.
    \end{definition}        
            
    \begin{lemma}\label{5.4.3}
       In the above settings, $ef\leq n$.
    \end{lemma}
    
    \begin{proof}
        Choose $\bar{R}$-linearly independent elements of $\bar{S}$: $\{u_i+Q\}_{i=1}^{f}$, for $u_i\in S$. This means if for $a_i\in R$ such that $\sum a_i u_i\in Q$, then $a_i\in P\subseteq Q$. 
        
        Let $b_1,\cdots, b_e$ be elements in $L^\times$ such that $|b_j||k^\times|$, $1\leq j\leq e$, are distinct elements in the group $|L^\times|/|k^\times|$. 
        
        We want to show $\{u_i b_j\}$, $1\leq i\leq f$, $1\leq j \leq e$ are linearly independent over $k$. If we multiply $b_j$ by elements in $k^\times$, we can make it in $Q$ so we may assume $b_j\in Q$.
        
        First we want to show if $a_i\in k$, $\sum a_i u_i\neq 0$ then $|\sum a_i u_i|\in |k^\times|$. Indeed, if $\sum a_i u_i\neq 0$, then $a_i\neq 0$ for some $i$. We may assume $0\neq |a_1|\geq |a_i|$, for all $a_i$. Then $|\sum a_i u_i |=|a_1||\sum a_1^{-1}a_i u_i|$ and $|\sum a_1^{-1}a_i u_i|\leq 1$. If $|\sum a_1^{-1}a_i u_i|<1$ then $\sum a_1^{-1}a_i u_i\in Q$ and since $|a_1^{-1}a_i|\leq 1$, $a_1^{-1}a_i\in R$. The non-trivial relation $\sum a_1^{-1}a_i u_i\in Q$ contradict with the choice of $u_i$. Hence $|\sum a_1^{-1}a_i u_i|=1$ and $|\sum a_i u_i|=|a_1|$ $\in |k^\times|$.
        
        To prove linear independence we assume $\sum_{i,j}a_{ij}u_ib_j=0$, $a_{ij}\in k$. We want to show $\sum_i a_{ij}u_i=0$ for all $j$. Otherwise, say $\sum_i a_{ij}u_i\neq 0$ for some $j$. Then by similar argument used in $\ref{5.3.2}$, we must have a pair $(j_1,j_2)$ such that $|\sum_i a_{ij_1}u_ib_{j_1}|$ $=|\sum_i a_{ij_2}u_ib_{j_2}|\neq 0$. Then $\sum_i a_{ij_1}u_i\neq 0$, $\sum_i a_{ij_2}u_i\neq 0$. Hence $|\sum_i a_{ij_1}u_i|\in |k^\times|$, $|\sum_i a_{ij_2}u_i|\in |k^\times|$ by our first claim in last paragraph. Then $|b_{j_1}||k^\times|=|b_{j_2}||k^\times|$, contradicts to the choice of $b_j$, thus $\sum_i a_{ij}u_i=0$ for all $j$. Now if we scale $a_{ij}$ by suitable elements in $k^\times$, we may assume $\sum_i a_{ij}u_i=0$ for all $a_{ij}\in R$. If $a_{ij}\neq 0$, then there is $0\neq |a_{{i_1}j}|\geq |a_{ij}|$ for all $i$. Then $|\sum_i a_{ij}u_i|=|a_{{i_1}j}||\sum_i a_{{i_1}j}^{-1}a_{ij}u_i|$ and $|\sum_i a_{{i_1}j}^{-1}a_{ij}u_i|\leq 1$. For the same reason as last paragraph $|\sum_i a_{{i_1}j}^{-1}a_{ij}u_i|=1$ but this gives $0=|\sum_i a_{ij}u_i|=|a_{{i_1}j}|\neq 0$ a contradiction. Therefore $a_{ij}=0$. Then $\{u_ib_j\}$, $1\leq i\leq f$, $1\leq j\leq e$, are linearly independent. This implies $ef\leq n$. 
    \end{proof}
            
    \begin{theorem}\label{5.4.4}
        Let $k$ be complete with respect to a discrete valuation $|\cdot |$. Let $L/k$ be a finite field extension. Let $|\cdot |$ be the unique extension of $|\cdot |$ to $L$. Then $ef=n$, $n=[L:k]$.
    \end{theorem}
        
    \begin{align*}
        &\Pi \quad&L&\geq S\idealinv Q& \quad &\bar{S}=S/Q&\quad e=[|L^\times|:|k^\times|]&\\
        &\quad  & |&&&&\\
        &\pi \quad&k&\geq R\idealinv P& \quad &\bar{R}=R/P&\quad f=[\bar{S}:\bar{R}]&\\
    \end{align*}
            
    \begin{proof}
        Let $P,\ Q,\ R,\ S,\ \bar{S},\ \bar{R},\ e,\ f$ as before, summarized in the above figure. Since $|\cdot |$ is discrete and $L$ is complete, choose $\pi\in P$, $\Pi\in Q$ such that $P=\pi R$, $Q=\Pi S$, $|\pi|$, $|\Pi|$ generates $|k^\times|$ and $|L^\times|$ respectively. Since $e=[|L^\times|:|k^\times|]$, we have $|\pi|=|\Pi|^e$, i.e., $\pi=u\Pi^e$ for some unit $u$ in $S$. Let $\{u_i+Q\}_{i=1}^f$ be a basis of $\bar{S}$ over $\bar{R}$, $u_i\in S$. We want to show that $$u_i\Pi^j, \quad 1\leq i\leq f, 0\leq j\leq e-1$$ form a basis of $L$ over $k$. Since $|\Pi^j||k^\times|$ are distinct elements of $|L^\times|/|k^\times|$, by \ref{5.4.3}, $u_i\Pi^j$ are linearly independent over $k$. It remains to show it is a spanning set for $L$. By multiplying elements of $k^\times$, we can scale every element of $L$ to $S$. Hence it suffices to show for every $v\in S$, it can be written in the form $v=\sum_{i,j}\alpha_{ij}u_i\Pi^j$ for $\alpha_{ij}\in R$. 
        
        Now consider $v\in S$, $|v|=|\Pi^k|$. Write $k=em+n$. $|v|=|\Pi^{em+n}|=|\pi^m\Pi^n|$. Hence $v=w\pi^m\Pi^n$ for some unit $w\in S$, $|w|=1$. By definition of $u_i$, $w+Q=\sum a_iu_i+Q$ for some $a_i\in R$. $w-\sum a_iu_i\in Q$ so $|w-\sum a_iu_i|<1$. Define $$v_1=v-(\sum_i a_iu_i)\pi^m\Pi^n=(w-\sum_i a_iu_i)\pi^m\Pi^n.$$ $|v_1|=|w-\sum_i a_iu_i||\pi^m\Pi^n|<|\pi^m\Pi^n|=|v|$. Then $v=\sum_i a_i\pi^m u_i\Pi^n+v_1=\sum_i b_iu_i\Pi^n+v_1$ where $b_i=a_i\pi^m\in R.$ Now we repeat this process and form $v_2$, $v_3$, $\cdots$ with $|v|>|v_1|>$ $|v_2|>\cdots$. Substituting into one equation gives $v=\sum_{i=1}^f \sum_{j=0}^{e-1} c_{ij}^{(m)}u_i\Pi^j+v_m$ for $m=1,2,\cdots$ and $c_{ij}^{(m)}\in R$. Then $v_m\to 0$ and $\sum_{i,j} c_{ij}^{(m)}u_i\Pi^j\to v$ as $m\to \infty$. Since $k$ is complete and $\pi_i\Pi^j$ are linearly independent, by Cauchy sequence argument in \ref{5.3.1} we know $lim_m c_{ij}^{(m)}$ exists for every $i,j$. Since $|c_{ij}^{(m)}|\leq 1$, we have its limit $|c_{ij}|\leq 1$. $c_{ij}\in R$. Hence $v=\sum c_{ij}u_i\Pi^j$ for every $c_{ij}\in R$. This proves the claim and hence proves the lemma.
    \end{proof}

    \section{Local Fields and their Finite Extensions}
        The section deals with local fields and structure of their finite extensions. We will show how the finite extensions of local fields look like and prove a existence and uniqueness theorem for unramified extensions. 
        
        To identify the structure, we need Hensel's lemma, which gives the reducibility criterion for polynomials with coefficients in a valuation ring. We will focus on the case when valuations are discrete, because local fields are.
        
        \begin{lemma}\label{5.5.1}
            Let $k$ be a field complete with respect to $|\cdot|$, $R$ be the corresponding valuation ring and $P$ its maximal ideal. $\bar{R}\defeq R/P$. Suppose $f(x)\in R[X]$ is monic, irreducible. Then its image $\bar{f}(x)\in \bar{R}[X]$ is a power of an irreducible polynomial in $\bar{R}[X]$.
        \end{lemma}
        
        \begin{proof}
        Let $L$ be a splitting field of $f(x)$. $|\cdot|$ be the unique extension on $L$. $S$ be the valuation ring and $Q$ be the maximal ideal of $S$. Now let $a\in L$, $\sigma\in Aut_kL$. Then by section \ref{5.3}, $$|\sigma(a)|=|N_{L/k}(\sigma(a))|^{1/[L:k]}|=|N_{L/k}(a)|^{1/[L:k]}|=|a|.$$ It follows that $\sigma(S)=S$, $\sigma(Q)=Q$. Every $\sigma$ induces an automorphism of $\bar{S}=S/Q$, $\bar{\sigma}:\bar{a}\mapsto \overline{\sigma a}$. $\bar{\sigma}(\bar{R})=\bar{R}$. Hence $\bar{\sigma}\in Aut_{\bar{R}}\bar{S}$. Now since $R$ is a P.I.D., hence a U.F.D., so $f(x)$ is irreducible in $k[X]$. Let $f(x)=\prod_{i=1}^d (x-r_i)$ in $L[X]$, where $d$ is the total number of roots counting multiplicity. Now let $a_d=f(0)=\prod_{i=1}^d (-1)^d r_i$ then $N_{L/k}(a)=[(-1)^da_d]^{[L:k]/d}$. Since $a_d$ is in $R$, $|r_i|\leq 1$ for all $i$, so $r_i\in S$. Passing to the image in $\bar{S}$, we have $\bar{f}(x)=\prod_{i=1}^n(x-\bar{r}_i)$. Let $\bar{r}_i$, $\bar{r}_j$ be any two roots. Since $f(x)$ is irreducible in $k[X]$, there exist an automorphism $\sigma\in Aut_kL$ such that $\sigma r_i=r_j$. Then passing to the image $\overline{\sigma r_i}=\bar{r}_j$. This implies $\bar{r}_i$ and $\bar{r}_j$ have the same minimal polynomial over $\bar{R}$, denoted as $\bar{g}(x)$. Then $\bar{f}(x)$ is some power of $\bar{g}(x)$.
        \end{proof}
            
        \begin{theorem}[Hensel's Lemma]\label{Hensels}
            Let $k$ be complete with respect to a discrete $|\cdot|$. $R$ be the corresponding valuation ring. $P$ is the maximal ideal. $\bar{R}=R/P$. Let $f(x)\in R[X]$ be monic, and $\bar{f}(x)=\bar{\gamma}(x)\bar{\delta}(x)$ in $\bar{R}[X]$, where $\bar{\gamma}(x)$, $\bar{\delta}(x)$ are monic, and $(\bar{\gamma}(x),\bar{\delta}(x))=1$, i.e., are coprime. Then $f(x)=g(x)h(x)$ in $R[X]$ where $g(x)$, $h(x)$ are monic and $\bar{g}(x)=\bar{r}(x)$, $\bar{h}(x)=\bar{\delta}(x)$.
        \end{theorem}
        
        \begin{proof}
        Factorize $f(x)$ into monic, irreducible factors $f_i(x)$: $f(x)=\prod_{i=1}^s f_i(x)^{e_i}$. By last lemma \ref{5.5.1}, $\bar{f}_i(x)=\bar{g}_i(x)^{k_i}$ for some monic, irreducible $\bar{g}_i(x)\in \bar{R}(x)$. $\bar{f}(x)=\prod_{i=1}^s \bar{g}_i(x)^{k_ie_i}$. Since         $\bar{f}(x)=\bar{\gamma}(x)\bar{\delta}(x)$ for coprime $\bar{\gamma}(x)$, $\bar{\delta}(x)$. We may assume $\bar{\gamma}(x)=\prod_{i=1}^r \bar{g}_i(x)^{k_ie_i}$, $\bar{\delta}(x)=\prod_{j=r+1}^{s} \bar{g}_j(x)^{k_je_j}$, where $\bar{g}_i(x)\neq \bar{g}_j(x)$ for $i,j$ in the indicated range. We may put $g(x)=\prod_{i=1}^r f_i(x)^{e_i}$, $h(x)=\prod_{j=r+1}^s f_j(x)^{e_j}$ giving the result.
        \end{proof}
            
        \begin{corollary}\label{5.5.3}
            In the settings above, if $\bar{f}(x)$ has a simple root $\bar{\alpha}$ in $\bar{R}[X]$, then it lifts to a root $r$ in $R$ such that $\bar{r}=\bar{\alpha}$.
        \end{corollary}
    
        \begin{proof}
            If $\bar{f}(x)=(x-\bar{\alpha})\bar{g}(x)$, $\bar{g}(x)\neq 0$, so $((x-\bar{\alpha}),\bar{g}(x))=1$ and we can apply Hensel's lemma and get the result.
        \end{proof}
            
        \begin{definition}[local field]\index{local field}
            A field $k$, together with absolute value $|\cdot|$ is called a \emph{local field} if it satisfies the following conditions:
            \begin{enumerate}[(i)]
                \item $|\cdot|$ is non-Archimedean, discrete, and non-trivial;
                \item $k$ is complete with respect to $|\cdot|$;
                \item The residue field of $|\cdot|$ is finite.
            \end{enumerate}
        \end{definition}
        
        \begin{remark}
            \begin{enumerate}[(i)]
                \item A typical example of local field is $(\Q_p, |\cdot|_p)$. Its valuation ring is $\Z_p$, the ring of p-adic integers. The maximal ideal is $p\Z_p$. The residue field is $\Z/p\Z$. Every element of $x\in \Q_p$ can be written uniquely in the form 
                $$x=\sum_{n=-N}^\infty a_np^n\quad for\ a_n\in \Z/p\Z,\ and\ some\ N\in \Z.$$
                Every element of $x\in \Z_p$ is uniquely written as $$x=\sum_{n=0}^\infty a_np^n,\quad a_n\in \Z/p\Z.$$
                Then $|x|_p=p^{-min\{n: a_n\neq 0\}}$;
                \item Conditions $(i)-(iii)$ carry over to finite extension of local fields by \ref{5.3}, \ref{5.3.2}, hence the finite extension of local field is local.
                
            \end{enumerate}
        \end{remark}
    
        \begin{lemma}\label{5.5.5}
            Let $k$ be a local field. $k\geq R\idealinv P$, $\bar{R}=R/P$ as before. Let $n_p=|\bar{R}|$. Then $R$ contains set $\Lambda_k$ of $n_p$ distinct roots to the equation $x^{n_p}=x$ and there is a group isomorphism $\Lambda_k^\times\cong \bar{R}^\times$.
        \end{lemma}
    
        \begin{proof}
            Let $n_p=|\bar{R}|$ and $q=char(\bar{R})$. Let $\Z/q\Z$ be its prime field. $n_p$ is some power of $q$ and $\bar{R}$ is the splitting field of $x^{n_p}-x$ over $\Z/q\Z$ (c.f. p.287 \cite{BAI}). Now pick any $\bar{\xi}_0\in \bar{R}$. It is a simple root of $f(x)=x^{n_p}-x$. By \ref{5.5.3}, there exist $\xi_0 \in R$ such that $\xi_0+P=\bar{\xi}_0$, $f(\xi_0)=0$. Now if we have $\bar{\xi}_1\neq \bar{\xi}_0\in \bar{R}$, then we have $\xi_1\in R$ with $f(\xi_1)=0$ but $\xi_1+P=\bar{\xi}_1\neq \bar{\xi}_0=\xi_0+P$. Hence we obtain $n_p$ elements $\xi_i\in \Lambda_k$ and $\xi_i+P\in \bar{R}$ are distict. Now it is easy to verify that $(\xi_i\xi_j)^{n_p}=\xi_i\xi_j$, and $\xi_i\to \xi_i+P$ is a group isomorphism $\Lambda_k^\times\cong \bar{R}^\times$.
        \end{proof}
    
        \begin{lemma}\label{5.5.6}
            Any finite extensions of finite fields are (Galois) cyclic extension.
        \end{lemma}
        
        \begin{proof}
            Let $k=\F_p$ be finite field of $p$ elements. Any finite extension is a finite field containing $k$, say $\F_{p^n}$ for some $n$. Since $\F_{p^n}$ is the splitting field of $x^{p^n}-x$ and $\F_{p}$ is perfect, $\F_{p^n}$ is Galois, with $|\F_{p^n}:\F_p|=|Gal(\F_{p^n}/\F_p)|=n$. Now consider the Frobenius automorphism $\sigma: \F_{p^n}\to \F_{p^n}$, $a\mapsto a^p$. Since $a^p-a=0$ for all $a\in \F_p$, $\sigma$ fixes $\F_p$. $\sigma$ is a bijection because $\F_{p^n}$ is finite, and a field. Therefore, $\sigma\in Gal(\F_{p^n}/\F_p)$. We now want to show $\sigma$ has order $n$.
            
            Let $\sigma^r:a\mapsto a^{p^r}$. If $\sigma^r=Id\in Gal(\F_{p^n}/\F_p)$, then $a_{p^r}=a$ for all $a\in \F_{p^n}$. But the polynomial $a^{p^r}-a$ of degree $p^r$ can have at most $p^r$ roots in $\F_{p^n}$. This forces $p^n\leq p^r$ and $n\leq r$. Hence $\sigma$ has order $n$ and generate $Gal(\F_{p^n}/\F_p)$.
        \end{proof}
    
        We will use this setting until the end of the section to determine the structure of finite dimensional extension fields of local field.
        
        Let $L/k$ be a finite dimensional local field $(k,|\cdot|)$, so $L$ is local. $|\cdot|$ is the given evaluation of the local field and extended unqiuely to the evaluation on $L$ (section \ref{5.3}). Let $e,f$ be the ramification index and the residue degree of $L/k$ with respect to $|\cdot|$. Let $R$ be the valuation ring of $k$. $P$ be the maximal ideal of $R$. $S$ is the valuation ring of $L$, and $Q$ is the maximal ideal of $L$. $\bar{S}\defeq S/Q$, $\bar{R}=R/P$. 
    
        \begin{align*}
       &(L,|\cdot|)\geq S\idealinv Q& \quad &\bar{S}=S/Q&\quad e=[|L^\times|:|k^\times|]&\\
        & |&&&\\
       &(k,|\cdot|) \geq R\idealinv P& \quad &\bar{R}=R/P&\quad f=[\bar{S}:\bar{R}]&\\
        \end{align*}
    
        Now we want to show $L/k$ is built up in two steps: $L\supseteq W\supseteq k$ where $W$ is unramified over $k$ and $L$ is completely ramified over $W$.
    
        \begin{theorem}\label{5.5.7}
            In the above setting, $k$ is a local field. Let $n_p=|\bar{R}|$ and $n_q=|\bar{S}|=n_p^f$. Let $\Lambda_k$ and $\Lambda_L$ be the set of roots of $x^{n_p}-x$, $x^{n_q}-x$. Let $W=k[\Lambda_L]$. Then $[W:k]=f$ and $W/k$ is unramified.
        \end{theorem}

        \begin{proof}
            By \ref{5.5.5}, there is an isomorphism $\Lambda_L^\times\to \bar{S}^\times$, $\xi\mapsto \bar{\xi}\defeq \xi+Q$. If $\xi$ is a primitive $(n_q-1)$-root of unity in $L$, then $\bar{\xi}$ is a primitive $(n_q-1)$-root of unity in $\bar{S}$. Since $\bar{S}$ is a finite field, $\bar{S}^\times$ is cyclic, hence $\Lambda_L^\times$ is cyclic, generated by $\bar{\xi}$, $\xi$ respectively. Therefore, $W=k(\xi)$, $\bar{S}=\bar{R}(\bar{\xi})$. Let $\bar{g}(x)$ be the minimal polynomial of $\bar{\xi}$ in $\bar{R}[X]$. Hence $deg( \bar{g})=[\bar{S}:\bar{R}]=f$. Moreover, $\bar{\xi}$ is a root of $x^{n_q}-x$. Hence we can factorize $x^{n_q}-x=\bar{g}(x)\bar{h}(x)$ in $\bar{R}[X]$ for some $\bar{h}(x)$. By Hensel's lemma \ref{Hensels}, they lift to $x^{n_q}-x=g(x)h(x)$ in $R[X]$ where $g(x)$, $h(x)$ is monic, with images in $\bar{R}[X]$ coincide with $\bar{g}(x)$, $\bar{h}(x)$. If $g(\xi)\neq 0$ then $h(\xi)=0$, $\bar{h}_0(\bar{\xi})=0$ but $\bar{\xi}$ is also a root of $\bar{g}(x)$. This constradicts with the fact that $x^{n_q}-x$ is separable. Hence $g(\xi)=0$. Since $\bar{g}(x)$ is irreducible in $\bar{R}[X]$, $g(x)$ is irreducible in $R[X]$. Hence $g(x)$ is irreducible in $k[X]$. Now $g(x)$ is the minimal polynomial of $\xi$ in $k[X]$. $[W:k]=deg(g)=f$. Since $\Lambda_L\subseteq W$, by the isomorphism $\Lambda_L^\times\cong \bar{S}^\times$ we get that the residue field of $W$ relative to the valuation is isomorphic to $\bar{S}$. Therefore the residue degree of $W/k$  is $f$. Therefore the ramification index is $1$ and $W$ is unramified.
        \end{proof}
    
        \begin{theorem}\label{5.5.8}
        In the setting of \ref{5.5.7}, suppose $T$ is the valuation ring of $W$. $O$ is its maximal ideal. Then $Gal(W/k)\cong Gal(\bar{T}/\bar{R})$ via the map $\sigma \mapsto \bar{\sigma}$, where $\bar{\sigma}: \bar{a}\mapsto \overline{\sigma a}$. Moreover, $Gal(W/k)$ is cyclic.
        \end{theorem}
        
        \begin{proof}
            Now $W$ is the splitting field of $x^{n_q}-x$ over $k$, which is separable, so $W/k$ is Galois. In the proof of \ref{5.5.1}, we see that every $\sigma\in Gal(W/k)$ induces a map $\bar{\sigma}:\bar{a}\mapsto \overline{\sigma a}$, $a\in \Lambda_L \in S\cap W=T$ by \ref{5.5.5}. Indeed, $\sigma(T)=T$, $\sigma(O)=O$ (\ref{5.5.1}). Therefore $\bar{\sigma}\in Gal(\bar{T}/\bar{R})$. Since the map $a\mapsto \bar{a}$ of $\Lambda_L^\times$ is injective, and $W=k[\Lambda_L]$, $\sigma\mapsto \bar{\sigma}$ is a monomorphism. Since $$|Gal(\bar{T}/\bar{R})|=[\bar{T}:\bar{R}]=[\bar{S}:\bar{R}]=f=[W:k]=|Gal(W/k)|,$$  following from $\bar{T}\cong \bar{S}$ as in the proof of \ref{5.5.7}, this map is an isomorphism. $Gal(W/k)$ is a cyclic group because $Gal(\bar{T}/\bar{R})$ is cyclic by \ref{5.5.6}.
        \end{proof}
    
        Now we have $W/k$ unramified, $|W^\times|=|k^\times|$, and $$[|L^\times|:|W^\times|]=e=n/f=[L:k]/[W:k]=[L:W].$$ $L$ is completely ramified over $W$.
        
        \begin{definition}[Frobenius automorphism]\index{Frobenius automorphism}
            In the above setting, $Gal(\bar{T}/\bar{R})$ is generated by $\bar{\sigma}_0: x\mapsto x^{n_p}$, where $n_p=|\bar{R}|$ (\ref{5.5.6}). Since $Gal(W/k)\cong Gal(\bar{T}/\bar{R})$ via the map $\sigma\mapsto \bar{\sigma}$ (\ref{5.5.8}). The element $\sigma_0\in Gal(W/k)$ which maps to $\bar{\sigma}_0$ generates $Gal(W/k)$. This $\sigma_0$ is called the \emph{Frobenius automorphism} of $W/k$.
        \end{definition}
    
        Now we want to derive an existence and uniqueness theorem for unramified extensions. \ref{5.5.7} already shows that such unramified extension exist. We want to show once the degree of extension $n=[W:k]$ is fixed, the unramified $W$ is unique up to isomorphism. First we need an analogue of "Eisenstein's criterion".
        
        Let $\Lambda_k=\{\xi_1,\cdots,\xi_{n_p}\}$ be the collection of distinct roots of $x^{n_p}=x$ in $R$ as before. Let $\pi$ be the generator of $P=\pi R$. Then $|\pi|$ generates $|k^\times|$. Let $\pi_k$ be the element such that $|\pi_k|=|\pi|^k$. We can choose $\pi_k=\pi^k$.
        
        For every $a\in k^\times$, we claim
        \begin{equation}\label{5.5.9.1}
            a=\alpha_{k_1}\pi_{k_1}+\alpha_{k_2}\pi_{k_2}+\cdots
        \end{equation}
        where every $\alpha_i\in \Lambda_k$, $k_1<k_2<\cdots$, and $\alpha_k\neq 0$. Let $k_1$ be such that $|a|=|\pi|^{k_1}=|\pi_{k_1}|$. Then $a\pi_{k_1}^{-1}\in R\setminus P$. There exist $\alpha_{k_1}\neq 0$ in $\Lambda^k$ such that $a\pi_{k_1}^{-1}\equiv \alpha_{k_1}\ mod\ P$ by \ref{5.5.5}. Then $1-a^{-1}\alpha_{k_1}\pi_{k_1}\in P$, $$|a-\alpha_{k_1}\pi_{k_1}|=|a||1-a^{-1}\alpha_{k_1}\pi_{k_1}|<|a|.$$ 
        
        If $a=\alpha_{k_1}\pi_{k_1}$ then it satisfies (\ref{5.5.9.1}). Otherwise we may continue this process with $a-\alpha_{k_1}\pi_{k_1}$ and obtain $k_1<k_2<\cdots$, $\alpha_{k_1},\alpha_{k_2},\cdots$ non-zero in $\Lambda_k$ such that $$|a|>|a-\alpha_{k_1}\pi_{k_1}|>|a-\alpha_{k_1}\pi_{k_1}-\alpha_{k_2}\pi_{k_2}|>\cdots$$
        and obtain (\ref{5.5.9.1}). It is clear that $\alpha$'s are uniquely determined. $a\in R$ if and only if all $k\geq 0$.
        
        \begin{definition}[Eisenstein polynomial]\label{5.5.10}\index{Eisenstein polynomial}
        Let $k_0$ be a subfield of a local field $k$ with
        \begin{enumerate}[(i)]
            \item $\Lambda_k\subseteq k_0$;
            \item $k_0$ is closed in the topology of $k$;
            \item $k_0\cap P\neq 0$.
        \end{enumerate}
        where $P$ is the maximal ideal of the valuation ring $R$ of $k$ in the usual way. Let $P_0=P\cap k_0$, $R_0=R\cap k_0$, then we call the polynomial $$f(x)=x^n+b_1x^{n-1}+\cdots+b_n\in R_0[X]$$ an \emph{Eisenstein polynomial} in $R_0[X]$ if $b_i\in P_0$ for all $i$, $b_n\notin P_0^2$.
        \end{definition}
    
        \begin{remark}
             It is clear that a subfield $k_0$ satisfying $(i)-(iii)$ is a local subfield of the local field $k$.
        \end{remark}
    
        \begin{lemma}\label{5.5.11}
            Let $k_0$ be a subfield of $k$ satisfying three conditions in \ref{5.5.10}. Let $\pi$ be the generator of $P=\pi R$. Then $k=k_0(\pi)$ and $\pi$ is algerbaic over $k_0$ with minimal polynomial an Eisenstein polynomial over $R_0=R\cap k_0$.
        \end{lemma}
    
        \begin{proof}
            Let $P_0=P\cap k_0$, $P_0=\pi_0 R_0$, and $P=\pi R$. Then $|\pi_0|=|\pi|^e$ for some $e\geq 1$. Let $k=eq+r$ with $0\leq r<e$. If $\pi_k=\pi_0^q\pi^r$, $|\pi_k|=|\pi|^k$.
            
            Now if $a\in R$ we have $k\geq 0$ in (\ref{5.5.9.1}). Then it can be rewritten as $$a=a_0+a_1\pi_1+\cdots+a_{e-1}\pi_{e-1}$$ where $a_i=\sum_{q\geq 0}\alpha_q\pi_0^q$. Since $k\geq 0$, $a_i\in R_0$. Now $|a_i\pi^i|$ has the form $|\pi|^{qe+i}$, therefore for $i\neq j$, $0\leq i,j<e$, $|a_i\pi^i|\neq |a_j\pi^j|$, so if $\sum a_i\pi^i=0$ then $a_i=0$ for all $i$.
            
            Thus $\{1,\pi, \cdots,\pi^{e-1}\}$ is a basis of $k/k_0$. $k=k_0(\pi)$ and $\pi$ is algebraic, with minimal polynomial of degree $e$, say $f(x)=x^e+b_1x^{e-1}+\cdots+b_e$. Moreover, $N_{k/k_0}(\pi)=\pm b_e$ and $$|b_e|=|N_{k/k_0}(\pi)|=|\pi^e|=|\pi_0|,$$ by \ref{5.3.1}, so $b_e\in P_0\setminus P_0^2$.
            
            Now if $b_i\notin P_0$ for some $i$, then because $b_e\in P_0$, we can take out all $x$'s and write $\bar{f}=\bar{g}_0(x)x^j$ for some $j\geq 1$ and $(\bar{g}_0(x), x^j)=1$. By Hensel's lemma \ref{Hensels}, $f(x)$ is reducible and a contradiction. Hence $b_i\in P_0$ for all $i$. It follows that $f(x)$ is an Eisenstein polynomial. 
        \end{proof}
    
        \begin{theorem}[existence and uniqueness]\label{5.5.12}
            Let $k$ be a local field and $L/k$ be a finite extension. Then $L$ contains a unique maximal unramified subfield $W$ satisfying the following properties.
            \begin{enumerate}[(i)]
                \item $[W:k]=f$, $f$ is the residue degree of $L/k$;
                \item $W$ is a cyclic extension over $k$;
                \item $L$ is completely ramified over $W$ and $[L:W]=e$, $e$ is the ramification index of $L/k$;
                \item If $\Pi$ is chosen such that $Q=\Pi S$, where $Q$ is the maximal ideal of the valuation ring $S$ of $L$, then $L=W(\Pi)$ and the minimal polynomial of $\Pi$ over $W$ is Eisenstein over $T=S\cap W$.
                
            \end{enumerate}
        \end{theorem}
    
        \begin{proof}
            Theorem \ref{5.5.7} proves the existence of $W$ and $(i)$. Theorem \ref{5.5.8} proves $(ii)$. Since $W$ is unramified, $|W^\times|=|k^\times|$, hence $[|L^\times|:|W^\times|]=e$. By tower law and \ref{5.4.4}, $[L:W]=e$. $(iii)$ is proved. By \ref{5.5.11}, $L=W(\Pi)$ and the minimal polynomial of $\Pi$ in $W[X]$ is an Eisenstein polynomial. $(iv)$ is shown.
            
            To show the uniqueness and maximality, we want to show every unramified extension is a splitting field of $x^{n_p^l}-x$, for $l|f$ (adapted from \cite{pierce1982associative}). Let $W'/k$ be another unramified extension, and $W'$ contains $L$. Let $T'$, $O'$ be its valuation ring and maximal ideal. $\bar{T}'\defeq T'/O'$. Let $l\defeq [W':k]=|Gal(W'/k)|\stackrel{\ref{5.5.8}}{=}|Gal(\bar{T}'/\bar{R})|=[\bar{T}':\bar{R}]$. Let $n_p=|\bar{R}|$ as usual. Then $|\bar{T}'|=n_p^l.$ $\bar{T}'$ is the splitting field of $x^{n_p^l}=x$, similar to the argument in \ref{5.5.7}. Let $\bar{\xi}$ be a primitive $n_p^l$-th root of unity in $\bar{T}'$ so $\bar{T}'=\bar{R}(\bar{\xi})$. Let $\bar{g}(x)$ be a minimal polynomial of $\bar{\xi}$ over $\bar{R}$. Then $x^{n_p^l}-x=\bar{g}(x)\bar{h}(x)$ for some $\bar{h}(x)$ in $\bar{R}(x)$. By Hensel's lemma \ref{Hensels}, it lifts to $x^{n_p^l}-x=g(x)h(x)$ in $R[X]$ with the image in $\bar{R}[X]$ coinciding with $\bar{g}(x)$, $\bar{f}(x)$. By \ref{5.5.3} $\bar{\xi}$ lifts to a root $\xi$ satisfying $x^{n_p^l}-x$. If $g(\xi)\neq 0$ then $h(\xi)=0$. But this implies $\bar{h}(\bar{\xi})=0$ contradicts to the fact that $x^{n_p^l}-x$ is separable. Hence $g(\xi)=0$. $g(x)$ is irreducible because $\bar{g}(x)$ is irreducible in $\bar{R}[X]$. Hence $g(x)$ is the minimial polynomial of $\xi$ in $R[X]$. Therefore $l=[\bar{T}':\bar{R}]=[\bar{R}(\bar{\xi}):\bar{R}]=deg(\bar{g})=deg(g)=[k(\xi):k]\leq [W':k]=l.$
            Hence $W'=k(\xi)$ is the splitting field of $x^{n_p^l}-x$.
            Now $$l=[\bar{T}':\bar{R}]=\frac{[S/Q:\bar{R}]}{[S/Q:\bar{T}']}=\frac{|Gal(S/Q/\bar{R})|}{[S/Q:\bar{T}']}\stackrel{\ref{5.5.8}}{=}\frac{|Gal(W/k)|}{[S/Q:\bar{T}']}\stackrel{(i)}{=}\frac{f}{[S/Q:\bar{T}']}.$$
            Hence $l|f$.
            
            Since $W=k[\Lambda_L]$, every solutions to $x^{n_p^l}-x$ is in $\Lambda_L$. Hence $W'\subseteq W$. $W$ is maximal and $W$ is unique because it is the splitting field of $x^{n_p^f}-x$.
        \end{proof}
    
        \begin{corollary}
            If $k$ is a local field. Then for each $n\in \N$ there is a unique field $W$ such that $[W:k]=n$ and $W$ is unramified.
        \end{corollary}
        \begin{proof}
             This theorem follows directly from the proof of uniqueness and maximality and $(i)$ in theorem \ref{5.5.12}.
        \end{proof}
        
    \chapter{The Brauer Group of Local Fields}
    \numberwithin{theorem}{chapter}
    
        We are now ready to determine the Brauer group of local fields with the results of all previous chapters. It is mainly based on \cite{jacobson2009}, unless explicitly cited. 
        
        Let $k$ be a local field. We want to establish an isomorphism between $k^\times/N(W^\times)$ and $Br(W/k)$, where $W$ is an unramified extension of $k$.
        
    \begin{lemma}\label{6.1}
        Let $\F_{q^n}/\F_q$ be a finite field extension, with $q^n$ and $q$ elements respectively. Then every $a\in \F_q^\times$ is a norm of an element $b\in \F_{q^n}^\times$ that is not contained in any proper subfield of $\F_{q^n}$.
    \end{lemma}
    
    \begin{proof}
        The Frobenius automorphism $\sigma: x\mapsto x^q$ generates $G=Gal(\F_{q^n}/\F_q)$ (\ref{5.5.6}). Hence the norm is $N_{\F_{q^n}/\F_q}(b)=\prod_{\sigma\in G}\sigma(b)=bb^q\cdots b^{q^{n-1}}=b^{(q^n-1)/(q-1)}$.
        
        The kernel of the norm map is the set $\{x\in F_{q^n}: x^{(q^n-1)/(q-1)}=1\}$ and it has order $(q^n-1)/(q-1)$ because $\F_{q^n}^\times$ is cyclic. Hence the image has order $q-1=|\F_{q^n}^\times|$. The norm map is surjective. Then for any $a\in \F_q^\times$, it is the norm of $(q^n-1)/(q-1)$ elements in $\F_{q^n}^\times$.
        
        If $b$ contains in some proper subfields of $\F_{q^n}$, $b$ is in some maximal proper subfields, which has cardinality of $q^m$ where $m$ is a maximal proper divisor of $n$. Since $\F_{q^n}^\times$ is cyclic, for every $m$ which is a maximal proper divisor of $n$, there exist a unique maximal subfield. Therefore at most $T=\sum (q^m-1)$ elements in the maximal subfields, where the sum is over all $m$ that are maximal proper divisors of $n$.
        
        But definitely $T<\frac{q^n-1}{q-1}$, therefore there is a $b$, not contained in any subfield, with its norm value $a$.
    \end{proof}
    
    \begin{remark}
        If $b$ is not in any maximal subfield of $\F_{q^n}$, then $\F_{q^n}=\F_q(b)$, and the minimal polynomial of $b$ in $\F_q(b)$ has order $n$.
    \end{remark}
        
    \begin{lemma}\label{6.2}
        Let $W/k$ be a finite, unramified extension of local field. Let $T$, $O$ be valuation ring and its maximal ideal of $W$. $R$, $P$ be the valuation ring, maximal ideal of $k$. Then any element $a\in R\setminus P$ is a norm in $W$. 
    \end{lemma}
    
     \begin{align*}
        &W\geq T\idealinv O& \quad &\bar{T}=T/O&\\
        &|\\
        &k\geq R\idealinv P& \quad &\bar{R}=R/P&\\
    \end{align*}
    
    \begin{proof}
      Let $\bar{T}=T/O$, $\bar{R}=R/P$. $W$ is unramified so $[W:k]=n=[\bar{T}:\bar{R}]$. By lemma \ref{6.1}, every $\bar{a}\in \bar{R}$, $\bar{a}=N_{\bar{T}/\bar{R}}(\bar{b})$, for $\bar{b}\in \bar{T}$ such that the minimal polynomial of $\bar{b}$ over $\bar{R}$ has degree $n$. By \ref{5.5.8}, $Gal(W/k)\cong Gal(\bar{T}/\bar{R})$ by $\sigma\mapsto \bar{\sigma}$, $\bar{\sigma}:\bar{x}\mapsto\overline{\sigma x}$, $x\in T$. Hence for any $x\in T$, the norm $\overline{N_{W/k}(x)}=N_{\bar{T}/\bar{R}}(\bar{x})$. Now choose $b\in T$ such that its image in $\bar{T}$ is $\bar{b}$, that means $\overline{N_{W/k}(b)}=N_{\bar{T}/\bar{R}}(\bar{b})=\bar{a}$. Also for the minimal polynomial $\bar{f}(x)$ of $\bar{b}$ in $\bar{R}[X]$, it lifts to some polynomial $f(x)$ in $R[X]$. Since $deg(\bar{f)}=n$ it is also the characteristic polynomial of $\bar{b}$ in its regular matrix representation (\ref{4.2.3}), and the constant term is $(-1)^n N_{\bar{T}/\bar{R}}(\bar{b})=(-1)^n \bar{a}$. 
      
      Since $\bar{f}(x)$ is separable, by corollary \ref{5.5.3}, $\bar{b}$ lifts to a root $b$ such that its image in $\bar{R}$ is $\bar{b}$. Hence we may assume the constant term of $f(x)$ to be $(-1)^n a$. Since $\bar{f}(x)$ is irreducible in $\bar{R}[X]$, $f(x)$ is irreducible in $R[X]$, hence in $k[X]$, so $f(x)$ is the minimal polynomial of $b$ in $k[X]$. 
      
      Moreover its degree is $n$, so does the characteristic polynomial of regular matrix representation of $a$. Its constant term is $(-1)^na=(-1)^n N_{W/k}(b)$. Hence $a=N_{W/k}(b)$.
    \end{proof}

    \begin{theorem}
        In the setting of \ref{6.2}, for $W$ unramified, $k^\times/N(W^\times)\cong Br(W/k)$ by $\pi^sN(W^\times)\mapsto[(\sigma,w, \pi^s)]$ where $0\leq s<n$, $\pi\in P-P^2$, $|\pi|$ generates $|W^\times|=|k^\times|$, and $\sigma$ the Frobenius automorphism of $W/k$. Both are cyclic groups of order $n\defeq[W:k]$.
    \end{theorem}
    
    \begin{proof}
    Since $W$ is unramified over $k$, $|W^\times|=|k^\times|$. We can choose $\pi\in k^\times$ such that $\pi\in O-O^2$ and $|\pi|$ generates $|W^\times|$ as in \ref{5.2.5}. Then any element $w\in W$ is of the form $w=u\pi^s$ for some unit $u\in T\setminus O$, $s\in \Z$. Then $N_{W/k}(w)=N_{W/k}(u)\pi^{ns}$, $n=[W:k]$, and $N_{W/k}(u)\in R\setminus P$ by the extension of valuation formula in $\ref{5.3.1}$.
    
    Conversely if $v\in k$, $v=u'\pi^{ns}$ where $u'\in R\setminus P$. By \ref{6.2}, $u'$ is a norm in $W$, and $\pi^{ns}=N_{W/L}(\pi^s)$ so $v$ is a norm. Therefore the elements of $N(W^\times)$ is of the form $v=u\pi^{ns}$ for some unit $u$, $s\in \Z$. 
    
    Therefore $k^\times/N(W^\times)$ is a cyclic group of order $n$ with elements $\pi^sN(W^\times)$, $0\leq s<n$. Hence $Br(W/k)$ is cyclic of order $n$.
    
    The previous isomorphism now becomes $\pi^sN(W^\times)\mapsto[(\sigma,W,\pi^s)]$, where $0\leq s<n$, $\sigma$ is the generator of $Gal(W/k)$, which is the Frobenius automorphism.
    \end{proof}
    
    By a topological argument, one has an important result which we will not prove here.
    
    \begin{theorem}[Theorem 9.21, p.607 \cite{jacobson2009}]\label{6.4}
        Let $D$ be a finite dimensional central division algebra over a local field $k$. Then $D$ is a cyclic algebra $D=(\sigma, W, \gamma)$ over $k$ where $W$ is unramified and $\gamma$ is a generator of maximal ideal of the valuation ring $R$ of $k$.
    \end{theorem}
    
    Now we are ready to determine the Brauer group of a local field by combining all the results together.
    
    \begin{theorem}
    Let $k$ be a local field. Then $Br(k)\cong \Q/\Z$.
    
    \end{theorem}
    
    \begin{proof}
     By theorem \ref{6.4}, any finite dimensional division algebra $D$ has a unramified extension field $W/k$ as splitting field. Therefore for $[A]\in Br(k)$, $[A]\in Br(W/k)$ for some unramified $W$. $A\sim (\sigma, W, \pi^k)$ for some $0\leq k<n$, $n=[W:k]$ and $\sigma$ is the Frobenius automorphism. Once $W$ is fixed, $k$ is uniquely determined. Now we map $[A]\mapsto r=k/n$, $0\leq k<n$ and wish to show this is well-defined and independent of choice of $W$.
     
     From \ref{5.5.12}, we know $W$ is uniquely determined by $n=[W:k]$ up to isomorphism, and if $W'/k$ is unramified and $[W':k]=m$, then $W'$ can be embeded into $W$ as a subfield if and only if $m|n$.
     
     Now it suffices to show if $W'\subseteq W$ is splitting field of $A$, then the rational number remains the same no matter if it is determined by $W'$ or $W$. Now by a dimension argument we have the Frobenius automorphism $\bar{\sigma}$ of $W'/k$ with order $m$ as we did in \ref{4.2.6}. Using \ref{4.2.6}, $A\sim (\bar{\sigma},W',\pi^{km/n})$ so the rational number determined by $W'$ is $\frac{km}{n}/m=k/n$, which remain the same. The map is well-defined.
     
     To prove surjectivity, for $r=k/n$, with $0\leq r<1$, i.e., $0\leq k<n$, we may take $W$ unramified with $[W:k]=n$. Then $(\sigma,W,\pi^k)$ is central simple with $W$ splitting field and maps into $r$.
     
     If $A$, $B$ are two central simple algebras, we can choose $W$ unramified that is splitting field for $A$, $B$, and let $A\sim (\sigma, W, \pi^{k_1})$, $B\sim (\sigma, W, \pi^{k_2})$ for $0\leq k_1, k_2<n$. $A\otimes B\sim (\sigma, W, \pi^{k_1+k_2})\sim (\sigma, W, \pi^s)$ where $s/n=(k_1+k_2)/n\ mod\ \Z$ and $0\leq s<n$.
     
     It follows that the map $[A]\mapsto k/n+\Z$ is an isomorphism.
    \end{proof}

    \chapter*{Conclusion and Future work}\addcontentsline{toc}{chapter}{Conclusion and Future work}
        This dissertation finishes the first two main steps of determining $Br(\Q)$. It starts by defining the Brauer group of a field, making use of the classical results of finite dimensional central simple algebra. Then by proving the famous Frobenius theorem, we determined $Br(\R)\cong \Z/2\Z$. At the later stage we introduce cyclic algebras, which are concrete algebras we can work on and every finite dimensional central division algebras we focus on are of this form. Finally by introducing valuations, we get more understanding of the finite extensions of the local fields by ramification and then are able to determine the Brauer group of local fields.
        
        One possibility of the future work is, of course, to determine $Br(\Q)$. In order to determine $Br(\Q)$, we need $Br(\R)$ and $Br(k_\nu)$ for local fields $k_\nu$, which we already have the result. In a certain sense the "only" completions (finite and infinite) of $\Q$ are $\R$ and $\Q_p$, so it is necessary to find out $Br(\R)$ and $Br(\Q_p)$ first. By Albert-Brauer-Hasse-Neother theorem \cite{wikiBrauergroup}, for a global field $k$, there is a canonical exact sequence \cite{stackexchangeBrauerGroup}\cite{wiki:albertbrauerhassenoether}
    \begin{equation}\label{0.1}
        0\to Br(k)\to \oplus_\nu Br(k_\nu)\to \Q/\Z\to 0
    \end{equation}
    where $\nu$ is a place (c.f. p.579 \cite{jacobson2009}) of $k$ and $k_\nu$ is the completion of $k$ at $\nu$. Now we know $Br(\R)\cong \Z/2\Z$, $Br(k_\nu)\cong \Q/\Z$, so for $k=\Q$ the exact sequence (\ref{0.1}) becomes 
    \begin{equation}\label{0.2}
        0\to Br(\Q)\to \Z/2\Z \oplus_\nu \Q/\Z\to \Q/\Z\to 0
    \end{equation}
    and have \cite{stackexchangeBrauerGroup}
        $$Br(\Q)=\{(a,x): \ a\in \{0,1/2\},\ x\in \oplus_p \Q/\Z\ and\ a+\sum x_p=0\}.$$
        It is a natural continuation of this dissertation to work out the maps in (\ref{0.2}) and show the above results in details.
        
        Another possible route is to work on applications of this content in a bigger context, the class field theory, and apply it to number theory. For example the proof of the quadratic reciprocity law in number theory. These topics are covered in some texts in algebraic number theory, especially chapters in local and global class field theory, such as chapter VII in \cite{cassels1967algebraic}.

\printindex
\addcontentsline{toc}{chapter}{Index}
\addcontentsline{toc}{chapter}{References}
\renewcommand{\bibname}{References}
\bibliography{refs}        

\begin{thebibliography}{10}

\bibitem{non-comm}
Konstantin Ardakov.
\newblock {Noncommutative Rings}, 2018.
\newblock URL: \url{https://courses.maths.ox.ac.uk/node/view_material/36733}.
  Last visited on 2019/05/30.

\bibitem{artin1944rings}
E~Artin, C~Nesbitt, and R~Thrall.
\newblock Rings with minimum condition. university of michigan.
\newblock {\em Ann Arbor, Michigan}, 1944.

\bibitem{ash2010ANT}
Robert~B Ash.
\newblock {\em A course in algebraic number theory}.
\newblock Courier Corporation, 2010.

\bibitem{cassels1967algebraic}
John William~Scott Cassels and Albrecht Fr{\"o}lich.
\newblock {\em Algebraic number theory: proceedings of an instructional
  conference}.
\newblock Academic Pr, 1967.

\bibitem{drozd2012finite}
Yurj~A Drozd and Vladimir~V Kirichenko.
\newblock {\em Finite dimensional algebras}.
\newblock Springer Science \& Business Media, 2012.

\bibitem{jacobson2009}
Nathan Jacobson.
\newblock {\em Basic algebra II}.
\newblock Dover Publications, 2 edition, 2009.

\bibitem{BAI}
Nathan Jacobson.
\newblock {\em Basic algebra I}.
\newblock Courier Corporation, 2 edition, 2012.

\bibitem{lam1999lectures}
TY~Lam.
\newblock Lectures on rings and modules.
\newblock {\em Graduate Texts in Mathematics}, 189, 1999.

\bibitem{stackexchangeBrauerGroup}
Daniel Miller.
\newblock Brauer group of a field of rational numbers.
\newblock Mathematics Stack Exchange.
\newblock URL:https://math.stackexchange.com/q/591279 (version: 2013-12-03).

\bibitem{pierce1982associative}
Richard~S Pierce.
\newblock Associative algebras, volume 88 of.
\newblock {\em Graduate texts in mathematics}, 1982.

\bibitem{wiki:albertbrauerhassenoether}
{Wikipedia contributors}.
\newblock Albert-brauer-hasse-noether theorem --- {Wikipedia}{,} the free
  encyclopedia.
\newblock
  \url{https://en.wikipedia.org/w/index.php?title=Albert%E2%80%93Brauer%E2%80%93Hasse%E2%80%93Noether_theorem&oldid=782158721},
  2017.
\newblock [Online; accessed 30-August-2019].

\bibitem{wikiBrauergroup}
{Wikipedia contributors}.
\newblock Brauer group --- {Wikipedia}{,} the free encyclopedia.
\newblock
  \url{https://en.wikipedia.org/w/index.php?title=Brauer_group&oldid=912950720},
  2019.
\newblock [Online; accessed 30-August-2019].

\end{thebibliography}
\bibliographystyle{plain}  

\end{document}